\documentclass[11pt,english,oneside]{amsart}
\usepackage{babel}
\usepackage[unicode=true,%
pdfusetitle,%
bookmarks=true,%
bookmarksnumbered=false,%
bookmarksopen=false,%
breaklinks=false,%
pdfborder={0 0 1},%
colorlinks=true]{hyperref}
\usepackage{amsthm,amstext,amssymb,amsbsy,amsmath,amscd,amsfonts}
\usepackage[nobysame,abbrev,alphabetic]{amsrefs}
\usepackage{ifpdf}
\usepackage{enumerate}
\usepackage{color}
\usepackage{graphicx}
\usepackage[all]{xy}
\xyoption{rotate}
\usepackage{mathrsfs}
\DeclareMathAlphabet{\mathpzc}{OT1}{pzc}{m}{it}
\usepackage[inline]{%
} 

\newcommand{\spec}{\mathrm{Spec}}
\newcommand{\rmG}{\mathrm{G}}

\newcommand{\rmf}{\mathrm{f}}
\newcommand{\ii}{\mathrm{i}}
\newcommand{\Id}{\mathrm{Id}}

\newcommand{\rmg}{\mathrm{g}}
\renewcommand{\G}{\mathbf{G}}
\newcommand{\Gm}{\mathbb{G}_m} 
\newcommand{\GL}{\mathrm{GL}}
\newcommand{\PGL}{\mathrm{PGL}}

\newcommand{\SL}{\mathrm{SL}}
\newcommand{\SO}{\mathrm{SO}}

\newcommand{\PSO}{\mathrm{PSO}}
\newcommand{\lact}{\curvearrowright}

\newcommand{\SLdZ}{\SL_2(\Zz)}
\newcommand{\Zz}{\mathbb{Z}}

\newcommand{\Zl}{\Zz_\ell}

\newcommand{\Qq}{\mathbb{Q}}

\newcommand{\Qt}{\Qq^\times}
\newcommand{\Qv}{\Qq_v}
\newcommand{\Qp}{\Qq_p}
\newcommand{\Zp}{\Zz_p}
\newcommand{\whZ}{\widehat{\Zz}}

\newcommand{\whO}{\widehat{\mcO}}
\newcommand{\whorder}{\widehat{\ordernumberfield}}
\newcommand{\whZt}{\widehat{\Zz}^\times}
\newcommand{\resprod}{\mathop{{\prod}^{\mathbf{'}}}\limits}
\newcommand{\GAf}{\G(\Af)}
\newcommand{\GA}{\G(\Aa)}
\newcommand{\mfa}{\mathfrak{a}}

\newcommand{\vphi}{\varphi}

\newcommand{\compactopen}{K} 
\newcommand{\compactopenf}{K_{\mathrm{f}}}
\newcommand{\Kf}{K_{\mathrm{f}}}

\newcommand{\multigp}{\mathbb{G}_m} 
\newcommand{\torus}{\mathbf{T}}
\newcommand{\torusK}{\mathbf{T}_{\field}}
\newcommand{\Tt}{\mathbf{T}} 

\newcommand{\res}{\mathrm{res}}
\newcommand{\what}[1]{\widehat{#1}}

\newcommand{\Mat}{\mathrm{Mat}}
\newcommand{\End}{\mathrm{End}}
\newcommand{\Hom}{\mathrm{Hom}}
\newcommand{\Aut}{\mathrm{Aut}}

\newcommand{\quat}{\mathbf{B}}
\newcommand{\quatt}{\mathbf{B}^\times} 

\newcommand{\PBt}{\mathbf{PB}^\times} 
\newcommand{\Norm}{\mathrm{Nr}}
\newcommand{\Nr}{\mathrm{Nr}}
\newcommand{\Tr}{\mathrm{Tr}}
\newcommand{\red}{\mathrm{red}} 

\newcommand{\units}[1]{#1{}^\times}
\newcommand{\punits}[1]{\mathrm{P}#1{}^\times}
\newcommand{\unitsadelef}[1]{\widehat{#1}{}^\times}
\newcommand{\punitsadelef}[1]{\mathrm{P}\widehat{#1}{}^\times}

\newcommand{\places}{\mathcal{V}} 
\newcommand{\fiplaces}{\mathcal{V}_{\mathrm{f}}}
\newcommand{\adele}{\mathbb{A}}
\newcommand{\KA}{\field(\mathbb{A})}
\newcommand{\KAf}{\field(\Af)}

\newcommand{\KtAf}{\field^\times(\Af)}
\newcommand{\Aa}{\mathbb{A}}

\newcommand{\Aat}{\mathbb{A}^\times}
\newcommand{\adelef}{\mathbb{A}_{\mathrm{f}}}
\newcommand{\Af}{\mathbb{A}_{\mathrm{f}}}
\newcommand{\Aft}{\mathbb{A}^\times_{\mathrm{f}}} 

\newcommand{\whOrt}{\widehat{\orderfield}^\times}

\newcommand{\projchar}[1]{#1_{\mathrm{char}}}
\newcommand{\Gres}{\G_{\mathrm{char}}}
\newcommand{\Grescomp}[1]{\G_{{#1},\mathrm{char}}}

\newcommand{\Ell}{\mathrm{Ell}}
\newcommand{\CM}{\Ell^{\mathrm{cm}}} 
\newcommand{\EllC}{\Ell_\infty}
\renewcommand{\SS}{\Ell^{\mathrm{ss}}} 

\newcommand{\mfp}{\mathfrak{p}}
\newcommand{\ovorderp}{\ov\order_{\mfp}}

\newcommand{\field}{\mathbf{K}}

\newcommand{\Kt}{\field^\times}
\newcommand{\fieldt}{\mathbf{K}^\times}
\newcommand{\bfK}{\mathbf{K}}

\newcommand{\integersnumberfield}{\ordernumberfield_{\field}}
\newcommand{\ordernumberfield}{\mathscr{O}}
\newcommand{\order}{\mathscr{O}}

\newcommand{\orderfield}{\mathscr{O}} 

\newcommand{\primeideal}{\mathfrak{p}} 
\newcommand{\Cl}{\mathrm{Cl}} 
\newcommand{\Pic}{\mathrm{Pic}} 
\newcommand{\ringclassfield}[1]{\mathbf{H}_{#1}} 
\newcommand{\bfH}{\mathbf{H}}

\newcommand{\Res}{\mathrm{Res}} 
\newcommand{\disc}{\mathrm{disc}}
\renewcommand{\mod}{\, \mathrm{mod}\,} 
\newcommand{\vol}{\mathrm{vol}}
\newcommand{\av}[1]{\left|#1\right|}

\newcommand{\bfp}{\mathbf{p}}
\newcommand{\Hh}{\mathbb{H}}
\newcommand{\Cc}{\mathbb{C}}

\newcommand{\Ct}{\mathbb{C}^\times}

\newcommand{\Rr}{\mathbb{R}}
\newcommand{\Ff}{\mathbb{F}}
\newcommand{\Fp}{\mathbb{F}_p}

\newcommand{\bash}{\backslash}
\newcommand{\mcO}{\mathcal{O}}

\newcommand{\mcC}{\mathcal{C}}
\newcommand{\mcE}{\mathcal{E}}
\newcommand{\mcL}{\mathcal{L}}
\newcommand{\mcI}{\mathcal{I}}
\newcommand{\mscK}{\mathscr{K}}
\newcommand{\OK}{\ordernumberfield_\field}

\def\map#1#2#3#4{\begin{matrix}#1&\to &#2
\\#3 &\mapsto &#4
\end{matrix}}

\newcommand{\ov}[1]{\overline{#1}}

\definecolor{darkolivegreen}{rgb}{0.33, 0.42, 0.18}

\newcommand{\awnote}[1]{\marginpar{\color{darkolivegreen}\tiny [ALW] #1}}

\newcommand\rquot[2]{
  \mathchoice
  {
    \text{\raise0.5ex\hbox{$#1$}\big/\lower0.5ex\hbox{$#2$}}%
  }
  {
    #1\,/\,#2
  }
  {
    #1\,/\,#2
  }
  {
    #1\,/\,#2
  }
}

\newcommand\lquot[2]{
  \mathchoice
  {
    \text{\lower0.5ex\hbox{$#1$}\big\backslash\raise0.5ex\hbox{$#2$}}%
  }
  {
    #1\,\backslash\,#2
  }
  {
    #1\,\backslash\,#2
  }
  {
    #1\,\backslash\,#2
  }
}

\newcommand\lrquot[3]{
  \mathchoice
  {
    \text{\lower0.5ex\hbox{$#1$}\big\backslash\raise0.5ex\hbox{$#2$\!}\big/
      \lower0.5ex\hbox{\!\!$#3$}}%
  }
  {
    #1\,\backslash\,#2\,/\,#3
  }
  {
    #1\,\backslash\,#2\,/\,#3
  }
  {
    #1\,\backslash\,#2\,/\,#3
  }
}

\newcommand{\ra}{\rightarrow}

\newcommand{\cA}{\mathcal{A}}

\newcommand{\cC}{\mathcal{C}}

\newcommand{\bC}{\mathbb{C}}
\renewcommand{\C}{\mathbb{C}}
\newcommand{\bF}{\mathbb{F}}

\newcommand{\bP}{\mathbb{P}}
\newcommand{\bQ}{\mathbb{Q}}
\newcommand{\Q}{\mathbb{Q}}
\newcommand{\bR}{\mathbb{R}}
\newcommand{\R}{\mathbb{R}}

\newcommand{\bZ}{\mathbb{Z}}
\newcommand{\Z}{\mathbb{Z}}

\theoremstyle{plain}
\newtheorem{theorem}{Theorem}[section]
\newtheorem{proposition}[theorem]{Proposition}
\newtheorem{cor}[theorem]{Corollary}
\newtheorem{lemma}[theorem]{Lemma}
\newtheorem{definition}[theorem]{Definition}

\theoremstyle{remark}
\newtheorem{remark}[theorem]{Remark}

\newtheorem*{claim*}{Claim}
\newtheorem{example}[theorem]{Example}

\newif\ifdraft 
\newif\ifold

\numberwithin{equation}{section}
\numberwithin{figure}{section}

\begin{document}

\title[Simultaneous supersingular reductions]{Simultaneous supersingular reductions of\\ CM elliptic curves}

\begin{abstract}
We study the simultaneous reductions at several supersingular primes of elliptic curves with complex multiplication. 
We show -- under additional congruence assumptions on the CM order -- that the reductions are surjective (and even become equidistributed) on the product of supersingular loci when the discriminant of the order becomes large. 
This variant of the equidistribution theorems of Duke and Cornut-Vatsal is an(other) application of the recent work of Einsiedler and Lindenstrauss on the 
classification of joinings of higher-rank diagonalizable actions.
\end{abstract}

\author{Menny Aka}
\address{M.A. Departement Mathematik\\
ETH Z\"urich\\
R\"amistrasse 101\\
8092 Zurich\\
Switzerland}
\email{mennyaka@math.ethz.ch}

\author{Manuel Luethi}
\address{School of Mathematical Sciences\\
Tel Aviv University\\
Tel Aviv-Yafo\\
Israel}
\email{manuelluthi@mail.tau.ac.il}

\author{Philippe Michel}
\address{EPFL/MATH/TAN, Station 8, CH-1015 Lausanne, Switzerland }
\email{philippe.michel@epfl.ch}

\author{Andreas Wieser}
\address{M.A. Departement Mathematik\\
ETH Z\"urich\\
R\"amistrasse 101\\
8092 Zurich\\
Switzerland}
\email{andreas.wieser@math.ethz.ch}

\thanks{Grants Acknowledgements:
  M.L.~acknowledges the support of the SNF (grants 200021-178958 and P1EZP2\_181725) and of the ISF (grant 1483/16).
  Ph.\ M. is partially supported by a DFG-SNF lead agency program grant 200021L\_175755 and by the SNF grant 200021\_197045. 
  A.W. is partially supported by the SNF grant 200021\_178958.   
 \today
 }

\maketitle
\setcounter{tocdepth}{1}
\tableofcontents

\section{Introduction}

\subsection{Simultaneous reduction of CM elliptic curves}\label{sec:introduction-arithmetic}
Let $\field = \Q(\sqrt{D})$ be an imaginary quadratic number field with ring of integers $\order_{\field}$ and discriminant $D=D_\field<0$.

An elliptic curve over $\C$ is a smooth projective variety in $\bP^2_\C$ defined by an equation of the shape
\begin{align*}
E: y^2 + a_1 xy + a_3 y = x^3 + a_2 x^2 + a_4 x + a_6
\end{align*}
and equipped with the structure of an abelian group via the usual chord-tangent construction and the point $[0,1,0]$ as neutral element. 
We say that $E/\bC$ has complex multiplication (CM) by $\integersnumberfield$ if its ring of endomorphisms is isomorphic to $\integersnumberfield$.
Denote by $\CM_{\order_\field}$ the finite set of $\C$-isomorphism classes of such elliptic curves.

Given $E \in \CM_{\order_\field}$, its $j$-invariant $j(E)$ is an algebraic integer, cf.~\cite[Ch.~II, Thm.~6.1]{silverman-advanced}, and  $\bfK(j(E))\subset \bar{\Q} \subset \C$ is the Hilbert class field $\bfH_\field$ of $\field$, cf.~\cite[Ch.~II, Thm.~4.3]{silverman-advanced}. Consequently, $E$ is defined over $\bfH_\field$.

Given $p$ a fixed odd prime, we also fix throughout an embedding $\overline{\Q} \hookrightarrow \overline{\Q_p}$; in particular this embedding uniquely determines, for any field $\field$ as above, a prime ideal $\primeideal=\primeideal_{\bfH_\field}\subset \order_{\bfH_\field}$ above $p$.

Assume that $p$ is inert in $\field$. The curve $E$ (or rather an $\bfH_\field$-form of it) has good reduction $E\mod \primeideal$ at~$\primeideal$, cf.~\S\ref{sec:supersingularreduction}, and its $j$-invariant is $$j(E\mod \primeideal)=j(E) \mod \primeideal \in \bF_{p^2},$$
and the curve $E\mod \primeideal$ is supersingular, i.e.~its endomorphism ring is a maximal order in a quaternion algebra over~$\bQ$, namely in the unique quaternion algebra  $\quat_{p,\infty}$ ramified at $p$ and $\infty$.

We denote by $\SS_p$ the set of isomorphism classes of supersingular elliptic curves defined over $\overline{\bF_p}$.
This is a finite set of cardinality $\frac{p}{12}+O(1)$ \cite[Ch.~V, Thm.~4.1]{silverman-aec}.
We will be interested in studying the reduction map
\begin{align*}
\mathrm{red}_p: \map{\CM_{\order_\field}}{\SS_p}{E}{E\mod \primeideal}
\end{align*}
for various values of $p$.

By Deuring's lifting theorem \cite{deuring41}  every supersingular curve can be obtained as the reduction of some elliptic curve (with complex multiplication by some order of some imaginary quadratic field). 
A natural question then is whether the CM order can be taken to be maximal and if so, which are the possible orders. 
In \cite{michel04}, the third named author remarked  that by combining the works of Gross, Duke and Iwaniec \cite{Gross,duke88,Iw}, any $E_0 \in \SS_p$ may be lifted to some $E \in \CM_{\order_\field}$ as soon as $p$ is inert in $\field$ and $D$ is sufficiently large (depending on~$p$). 
In other terms, for $D$ large enough, the reduction map $\mathrm{red}_p$ is surjective. This result was subsequently reinterpreted and refined by several authors -- see for example \cite{elkiesonoyang}, \cite{yang08}, \cite{Kane09}.

In the present paper we are interested in a further refinement of this question: namely whether given several supersingular elliptic curves at several distinct primes there exists a {\em single} elliptic curve with CM by $\order_\field$ which is a lifting of all of them. More precisely, let $\mathbf{p} = \{p_1,\ldots,p_s\}$ be a finite set of distinct odd primes and for each $i$ a fixed embedding $\overline{\Q} \hookrightarrow \overline{\Q_{p_i}}$.
Let $\field$ be an imaginary quadratic field in which the primes $p_1,\ldots,p_s$ are all inert.
For each $i\in \{1,\ldots,s\}$ let $\primeideal_i=\primeideal_{i,\bfH_\field}$ be the prime ideal in $\bfH_\field$ determined by $\overline{\Q} \hookrightarrow \overline{\Q_{p_i}}$.
We now have a simultaneous reduction map
\begin{align*}
\mathrm{red}_{\mathbf{p}}: \map{\CM_{\order_\field}}{\prod_{i\leq s}\SS_{p_i}}{E}{(E\mod\primeideal_i)_{i\leq s}}.
\end{align*}
The question is whether this map is surjective or not.
We expect that this is the case as long as $D$ is large enough (depending on $\bfp$). A consequence of our main Theorem \ref{thm:main-elliptic} stated below is that this is the case at least under some additional congruence assumptions on $D$:

\begin{theorem}[Simultaneous lifting]
Let $q_{1},q_{2}$ be two distinct, odd primes. Let $p_{1},\ldots,p_{s}$ be distinct odd primes also distinct from $q_{1}$ and $q_{2}$. There is some $D_{0}\geq 1$ (depending on $\mathbf{p} = \{p_1,\ldots,p_s\}$ and $q_1,q_2$) with the following property: for any imaginary quadratic field $\field$ of discriminant $D$ such that
\begin{enumerate}
\item $|D|\geq D_0$,
\item each $p_i$ for $1\leq i\leq s$ is inert in $\field$, and
\item $q_{1}$ and $q_{2}$ are split in $\field$,
\end{enumerate}
the map $\mathrm{red}_{\mathbf{p}}$ is surjective. In other terms, for all $E_{i}\in\SS_{p_{i}}$, $i=1,\ldots,s$ there exists an elliptic curve $E\in\CM_{\order_\field}$ satisfying
\begin{equation*}
  E\mod\primeideal_i= E_{i}\quad(i=1,\ldots,s).
\end{equation*} 
\end{theorem}

In fact this surjectivity result admits a more precise formulation in the form of an equidistribution statement on the product of $\prod_{1 \leq i\leq s}\SS_{p_i}$. 

Given $E_{0}\in\SS_{p}$, we define $$w_{E_{0}}=\lvert\mathrm{End}^{\times}(E_{0})/\{\pm1\}\rvert.$$
 Eichler's mass formula states that
\begin{align}\label{eq:Eichler mass formula}
\sum_{E_{0}\in\SS_{p}}\frac{1}{w_{E_{0}}}=\frac{p-1}{12},
\end{align} 
see e.g.~\cite[Eq.~1.2]{Gross}.
 We then introduce the probability measure 
\begin{equation*}
  \nu_{p}=\tfrac{12}{p-1}\sum_{E_{0}\in\SS_{p}}\tfrac{1}{w_{E_{0}}}\delta_{E_{0}}.
\end{equation*}
\begin{remark}It is known that the product of the $w_{E_{0}}$ over the distinct isomorphism classes is a divisor of 12, cf.~\cite[Eq.~1.1]{Gross}, and hence asymptotically most weights are equal to $1$ so the above measure  
is almost uniform.
\end{remark}

Using the observation in~\cite{michel04} referred to previously, the third author established that for every $E_{0}\in\SS_{p}$ we have
\begin{equation*}
  \frac{\lvert\{E\in\CM_{\order_\field}:E\mod \primeideal \cong E_{0}\}\rvert}{\lvert\CM_{\order_\field}\rvert}\rightarrow\nu_{p}(E_{0})
\end{equation*}
as $D\to-\infty$ along the set of fundamental discriminants. In other terms, the push-forward of the uniform probability measure on $\CM_{\order_\field}$ by the map $\mathrm{red}_{p}$ converges to the measure $\nu_p$.

We will establish a similar equidistribution statement towards the product measure $\bigotimes_{i\leq s}\nu_{p_i}$; in fact we  include an additional archimedean equidistribution result.

Let $\EllC$ be the moduli space of all complex elliptic curves up to $\Cc$-isomorphism. This space is identified with the space of lattices in $\Cc$ up to $\Ct$-homothety, i.e.~the complex modular curve  
$$Y_0(1)=\SLdZ\bash \Hh,$$
via the map
$$[z]=\SLdZ.z\mapsto [\Lambda_z]:=[\Zz+\Zz.z]\mapsto[\Cc/\Lambda_z].$$
In this representation an elliptic curve $E$ with CM by $\order_\field$ is one for which the corresponding lattice $\Lambda_{z_E}$ satisfies
$$\End(\Lambda_{z_E}):=
\{z\in\Cc: z.\Lambda_{z_E}\subset \Lambda_{z_E}\}=\order_\field\subset\field\subset\Cc.$$
By this identification $\EllC$ is equipped with a probability measure $\nu_{\infty}$ corresponding to the normalized hyperbolic measure $\frac3\pi\frac{dxdy}{y^2}$. 

Let $\mathrm{red}_\infty$ denote the obvious injection
\begin{equation*}
  \mathrm{red}_\infty:\map{\CM_{\order_\field}}{\EllC}{E}{E(\Cc)}.
\end{equation*}
It was proven by Duke in \cite{duke88} that the push-forward of the uniform probability measure on $\CM_{\order_\field}$ converges to $\nu_{\infty}$ as $D\ra -\infty$.

Setting now $$\bfp=\{p_0=\infty\}\cup\{p_i: 1\leq i\leq s\}$$ we consider the map $\mathrm{red}_\bfp = (\mathrm{red}_{p_i})_{0 \leq i \leq s}$ and its image
\begin{equation*}
  \mathrm{red}_\bfp(\CM_{\order_\field})=\big\{\big(\red_{p_i}(E)\big)_{0\leq i\leq s}:E\in\CM_{\order_\field}\big\}\subset\Ell_\infty\times\prod_{i=1}^s\SS_{p_i}.
\end{equation*}
We prove
\begin{theorem}\label{thm:main-elliptic}
  Let $q_{1},q_{2}$ be distinct, odd primes which are distinct from $p_1,\ldots,p_s$. As $D\ra-\infty$  along the set of fundamental negative discriminants such that
  \begin{enumerate}
    \item each $p_i$ for $1\leq i\leq s,$ is inert in $\field=\Q(\sqrt D)$, and
    \item $q_{1},q_{2}$ are split in $\field$,
  \end{enumerate}
the push-forward of the counting probability measure on $\CM_{\order_\field}$ by $\mathrm{red}_\bfp$ equidistributes towards $\nu_\infty\otimes\nu_{p_{1}}\otimes\cdots\otimes\nu_{p_{s}}$.
 \end{theorem}

This equidistribution result is true more generally if instead of a sequence of maximal orders $\order_\field$ we consider sequences of quadratic orders $\order\subset \order_\field\subset\field$ and the sets $\CM_\order$ of elliptic curves with complex multiplication by $\order$ (again with suitable congruence condition); we refer to Theorem \ref{thm:main-ellipticorder} for the precise statement.

\begin{remark}\label{rem:remarkaftermainell}
In this extended setting, a version of this equidistribution result was proven by Cornut \cite{Cornut-Invent} using ideas of Vatsal \cite{Vatsal-Invent} for non-maximal orders whose discriminants are of the form $Dp^{2n}$ for $D<0$ a fixed fundamental discriminant, $p$ a fixed prime coprime to $D$ and $n\rightarrow +\infty$.
As in the present paper, these works make crucial use of dynamics and ergodic theory on locally  homogeneous spaces. A chief difference is that the above mentioned works use unipotent dynamics (i.e.~Ratner's theorems such as \cite{Ratner-joining})
while we also need to rely on the recent work of Einsiedler and Lindenstrauss on rigidity of joinings of higher rank diagonalizable actions~\cite{EL-joining2}.

An especially interesting case is that of orders whose discriminants are of the form $Dp^n$ for $D<0$ a fixed fundamental discriminant, $n\geq 1$ a fixed exponent, and $p\rightarrow +\infty$ along the primes coprime to $D$ (see Theorem~\ref{thm:main-ellipticorder}). Although the quadratic field $\field=\Qq(\sqrt D)$ is fixed, the fact that $p$ is varying also seems to require the use of rigidity of higher rank diagonalizable actions.
  As explained to us by Henri Darmon such a case could be interesting for the study of certain Euler systems. We will pursue other variants along this direction in future joint work with D.~Ramakrishnan.
\end{remark}

\begin{remark}
Theorem~\ref{thm:main-elliptic} can be upgraded in several directions: 

\begin{itemize}
	\item[--] The most immediate is by incorporating additional level structures; for instance by considering for $N\geq 1$, the space $Y_0(N)$ of isomorphism classes of pairs of elliptic curves $\vphi:E_1\mapsto E_2$ where $\vphi$ has a cyclic kernel of order $N$. 
Suppose that every prime divisor of $N$ splits in $\field$, then if $(\vphi:E_1\mapsto E_2)\in Y_0(N)(\Cc)$ is such that $E_1\in \CM_{\order_\field}$, then  $E_2\in \CM_{\order_\field}$ as well; the set,  $Y_{0}(N)_\field^{\mathrm{cm}}$ say, of such "CM" points  is then called the set of {\em Heegner points} (\cite[I. 3.]{Gross2}). Note that if, in addition, $N$ has two distinct prime factors $q_1,q_2$, the splitting condition is automatic.

\item[--]  In \cite{HMRL}, S. Herrero, R. Menares, and J. Rivera-Letelier  have established a refined version of the equidistribution of CM elliptic curves modulo an inert prime $p$ (more generally for any kind of prime) by describing how the CM curves equidistribute on $\SS({\Cc_p})$, the moduli space of elliptic curves over $\Cc_p$ with supersingular reduction. Their proof relies among other things on Duke's approach to prove equidistribution (i.e.~by bounding Fourier coefficient of an adequate theta series). In a forthcoming project with the authors of \cite{HMRL} we will combine both their and our methods to prove the joint equidistribution of $\CM_{\order}$  on the product of the moduli spaces $\SS({\Cc_{p_i}})$ for $i=1,\ldots,s$.

\item[--] These equidistribution theorems also generalize to equidistribution for CM points in the space of abelian varieties of dimension $2[F:\Qq]$,
admitting complex multiplication by a quaternion algebra defined over a fixed totally real number field $F$.  We refer to \cite{Vatsal-Invent,Cornut-Invent,CornutJetchev,CVDoc,CVLMS} for examples of work in this direction.
\end{itemize}

\end{remark}

\begin{remark}
Though seemingly artifical, the splitting assumption on $q_1,q_2$ in Theorem~\ref{thm:main-elliptic} allows the use of diagonalizable dynamics alluded to in Remark~\ref{rem:remarkaftermainell} (see the discussion after Theorem~\ref{thm:main-dynamic}). One may generalize Theorem~\ref{thm:main-elliptic} to treat for instance all imaginary quadratic fields $\field$ for which there are two split primes $q_1,q_2 \not\in \{2,p_1,\ldots,p_s\}$ below a fixed threshold $C>0$.
Furthermore, even assuming GRH one cannot remove the splitting assumption using the arguments in the present paper as the dynamical argument uses two fixed split primes that are not allowed to vary with the field.
\end{remark}


\subsection{Reduction to an equidistribution statement}\label{sec:intro-reductionstep}
As hinted at above, Theorem \ref{thm:main-elliptic} is a consequence of a dynamical statement in the context of locally homogeneous spaces. The first step is to realize $\red_{\bfp}(\CM_{\order_\field})$ as  (the projection of) an orbit of the adelic points of the torus
\begin{align*}
\Tt_{\field}=\Res_{\field/\Qq}\Gm/\Gm
\end{align*}
on a product of adelic locally homogeneous spaces. We describe this realization here.
Using this realization, Theorem \ref{thm:main-elliptic} is then obtained from an equidistribution result for such orbits, more specifically Theorem~\ref{thm:main-dynamic} which will be discussed in the next section.

We use several $\Q$-forms of $\PGL_{2}$, namely the projective unit groups $\G_i=\PBt_i$ for $i=0,\ldots,s$ of the quaternion algebras $\quat_0,\quat_1,\ldots,\quat_s$ defined over $\Qq$
where $\quat_0=\mathrm{Mat}_2$ is the split quaternion algebra and where $\quat_{i}$ for $1\leq i\leq s$ is the quaternion algebra ramified exactly at $\infty$ and $p_i$.

As explained in \S\ref{sec:supersingular} there is for each $p_i$ an identification
\begin{align*}
\psi_i:\SS_{p_i} \overset{\sim}{\longrightarrow}\lrquot{\G_i(\Qq)}{\G_i(\adele)}{\G_i(\Rr)\punitsadelef{\mcO_i}}
\end{align*}
where $\mathcal{O}_i \subset \quat_i(\Q)$ is some fixed maximal order and $\widehat{O_i} = \mathcal{O}_i \otimes \widehat{\Z}$.
There is also an identification for the archimedean place
\begin{align*}
\psi_0:\Ell_\infty \overset{\sim}{\longrightarrow}\lrquot{\PGL_2(\Qq)}{\PGL_2(\adele)}{\mathrm{PSO}_2(\Rr)\PGL_2(\whZ)}.
\end{align*}
Let $E\in\CM_{\order_\field}$. 
To $E$ is associated for each $i=0,\ldots, s$ a class $[g_{E,i}] = \psi_i(\red_{p_i}(E))$ and an embedding of $\Qq$-algebras $
\iota_i:\bfK\hookrightarrow \quat_{i},
$
and consequently a morphism of $\Qq$-algebraic groups
$$\iota_i:\Tt_\field\hookrightarrow \G_i.$$

Let us recall that the ideal class group $\Pic(\OK)$ acts simply transitively on $\CM_{\order_\field}$; we denote this action by $$([\mfa], E)\mapsto [\mfa]\star E$$
where $[\mfa]$ is the class of an ideal $\mfa\subset \order_\field$. We also have a natural identification
\begin{align*}
\Pic(\order_\field)\simeq \lrquot{\Tt_\field(\Qq)}{\Tt_\field(\adele)}{\Tt_\field(\Rr)\Tt_\field(\whZ)}
\end{align*}
where $\Tt_\field(\whZ)$ corresponds to the projectivized group of units in the completion $\widehat{\order_{\field}}$.
where the integral structure defining $\Tt_\field(\whZ)$ is the one induced by the choice of maximal orders $\mathcal{O}_1,\ldots, \mathcal{O}_s$.
We show in \S\ref{sec:red_infty} and \S\ref{sec:comp Picard action} that for $i=0,\ldots, s$, if $t_\mathrm{f}\in\torus_\field(\adelef)$ is such that $[t_\mathrm{f}]\simeq[\mfa]$ then
\begin{align*}
\psi_i(\red_{p_i}([\mfa]\star E))=[\iota_i(t_\mathrm{f}^{-1}) g_{E,i}].
\end{align*}

We now collect these identities for all $0 \leq i\leq s$. 
Let us set 
\begin{gather*}
\iota = (\iota_i)_{0 \leq i\leq s}: \Tt_\field \to \G=\prod_{i=0}^s\G_i
\end{gather*}
and define $\rmg_E = (g_{E,i})_{0 \leq i\leq s}$ as well as the compact groups
\begin{align*}
K_\infty=\PSO_2(\Rr)\times\prod_{i=1}^s \G_i(\Rr)
\quad \text{and} \quad
\Kf=\PGL_2(\whZ)\times\prod_{i=1}^s \mathrm{P}\what\mcO_i^\times.
\end{align*}
We obtain that the componentwise identification
\begin{align*}
\psi:\Ell_\infty\times\prod_i\SS_{p_i}\overset{\sim}{\longrightarrow} \lrquot{\G(\Qq)}{\G(\Aa)}{K_\infty \Kf}
\end{align*}
satisfies for any $[t_\mathrm{f}]$ and $[\mfa]$ as above
\begin{align}\label{eq:comp equivariance}
\psi(\red_\bfp([\mfa]\star E))=[\iota(t_\mathrm{f}^{-1})\rmg_E].
\end{align}
Therefore $\red_\bfp(\CM_{\order_\field})$ is identified with the projection of the orbit of the torus $\torus_\field(\Aa)$ embedded diagonally into the product group $\G$ by $\iota$: 
\begin{align*}
[\iota(\Tt_\bfK(\adele)) \rmg_E]\subset\lrquot{\G(\Qq)}{\G(\Aa)}{K_\infty\Kf}.
\end{align*}
Let $\torus_\iota$ denote the image $\iota(\torusK)$.
The situation can be summarized in the following commutative diagram:
\begin{equation*}
  \begin{gathered}
    \xymatrix{
     [\mfa]\star E \in \Pic(\OK).E = \CM_{\order_\field}
     \ar[r]^(.57){\red_\bfp}
     \ar[d]_{\iota} 
     & \Ell_\infty\times\prod_{i=1}^{s}\SS_{p_{i}}
     \ar[d]^{\psi}\\
      [\iota(t_\mathrm{f}^{-1})] \in \displaystyle{\lrquot{\Tt_\iota(\Qq)}{\Tt_\iota(\adele)}{\Tt_\iota(\R)\Tt_\iota(\whZ)}}
      \ar[r]
      &
      \displaystyle{\lrquot{\G(\Qq)}{\G(\Aa)}{K_\infty \Kf}}
    }
  \end{gathered}
\end{equation*}
Here, the bottom arrow is given by mapping $[\iota(t)]$ to $[\iota(t)\rmg_E]$ seen as an element in the double quotient on the right.
This discussion reduces Theorem~\ref{thm:main-elliptic} to considering the behaviour of sets of the form $[\iota(\Tt_\bfK(\adele)) \rmg_E]$ as the discriminant of $\field$ goes to negative infinity.

\subsection{Equidistribution of diagonal torus orbits}\label{sec:introductiondynamics}

By the reduction step in the previous section, the proof of Theorem \ref{thm:main-elliptic} is a direct consequence of a general equidistribution theorem. 
To state it in full, we (re-)introduce some notation.

Let $\quat_{i}$ for $\ 0\leq i\leq s$ be finitely many distinct rational quaternion algebras.
We denote by $\G_i=\PBt_i$ the associated projective group of units and set
\begin{align*}
\G = \prod_{i=0}^s \G_i
\quad \text{and}\quad
[\G]=\G(\Qq)\bash\G(\Aa).
\end{align*}
Furthermore, we fix for each $i \in \{0,\ldots,s\}$ a compact open subgroup $\compactopen_{\mathrm{f},i} < \G_i(\adelef)$
 and define 
 \begin{equation}\label{eq:compactopen}\compactopenf=\prod_i \compactopen_{\mathrm{f},i}\hbox{ and }[\G]_{\compactopenf}=[\G]/\compactopenf.	
 \end{equation}
By a theorem of Borel and Harish-Chandra \cite{borelharishchandra}, the homogeneous space $[\G]$ comes with a Haar probability measure which we denote by $dg$.
It induces a probability measure on $[\G]_{\compactopenf}$ (by restriction to the continuous right $\compactopenf$-invariant functions on $[\G]$) which we also denote by $dg$ for simplicity.

\subsubsection{Toral packets and discriminants}

For each $i\in \{0,\ldots,s\}$ we let $\iota_i$ be an embedding  of a quadratic field $\field$ (possibly a real quadratic field) into the quaternion algebra $\quat_i(\bQ)$. 
We thus obtain an induced morphism of $\bQ$-algebraic groups
\begin{equation*}
\iota_i:\torus_\field \to \G_i
\end{equation*}
and a diagonal morphism $$\iota:\torus_\field \to \G.$$
We denote its image by
\begin{equation}\label{eq:def diagonal torus}
\torus_\iota = \big\{(\iota_0(t),\ldots,\iota_s(t)): t \in \operatorname{Res}_{\field/\bQ}(\multigp)/\multigp\big\}.
\end{equation}

For $\rmg = (g_0,\ldots,g_s)$ in $\G(\adele)$ and a tuple $\iota=(\iota_{0},\ldots,\iota_{s})$ of embeddings as above, we denote by
\begin{align*}
[\torus_\iota \rmg] = \G(\Qq)\bash\G(\bQ) \torus_\iota(\adele)\rmg \subset [\G]
\end{align*}
the associated (compact) toral packet and by $[\torus_\iota \rmg]_{\compactopenf}$ its projection to $[\G]_{\compactopenf}$. 

The arithmetic complexity of the toral packet $[\torus_\iota \rmg]$ is measured using the notion of {\em discriminant}; in the present case we define the discriminant as
\begin{align}\label{eq:discproduct}
  \disc([\torus_\iota \rmg]) = \min_{i=0,\ldots,s} \disc([\torus_{\iota_i} g_i]).
\end{align}
where $[\torus_{\iota_i} g_i]$ is the projection of $[\torus_\iota \rmg]$ to the $i$-th factor $[\G_i]$ and the discriminant $\disc([\torus_{\iota_i} g_i])$ is defined in
 \cite[\S4]{dukeforcubic}.

\begin{remark}
In the context of Theorem~\ref{thm:main-elliptic} and in view of Section~\ref{sec:intro-reductionstep}, the above discriminant of a toral packet $[\torus_\iota \rmg]$ with associated field $\field$ is comparable to $|\disc(\field)|$ and the terms on the right-hand side of \eqref{eq:discproduct} are roughly the same size.
\end{remark}
 
\subsubsection{The main theorem}

We are now able to formulate the dynamical input for Theorem \ref{thm:main-elliptic}, which is  of independent interest. 

\begin{theorem}[Equidistribution of diagonal torus orbits]\label{thm:main-dynamic}
Let $q_1,q_2$ be two distinct odd primes. Let $\torus_{n}=\torus_{\iota_n}<\G$ for $n\geq 1$ be a sequence of tori as in \eqref{eq:def diagonal torus} with underlying fields $\field_n$ such that  $q_1$ and $q_2$ both split in $\field_n$ for every $n$.
For any $n \geq 1$ let $\rmg_n\in \G(\adele)$.

Let $\compactopenf < \G(\adelef)$ be a compact open subgroup as in \eqref{eq:compactopen}.
We assume that for every $i=0,\ldots,s$ the reduced norm $\Norm_{\quat_i}$ induces a surjective map
\begin{align}\label{eq:assumption compactopen}
\Norm_{\quat_i}: \compactopen_{\mathrm{f},i} \twoheadrightarrow \rquot{\widehat{\Z}^{\times}\!\!}{(\widehat{\Z}^{\times})^2}.
\end{align}

If $\disc([\torus_n \rmg_n]) \to \infty$ as $n \to \infty$
the sequence of packets $[\torus_n \rmg_n]_{\compactopenf}$ equidistributes on $[\G]_{\compactopenf}$.
That is, for any $\compactopenf$-invariant function $f \in \mcC_c([\G])$ we have
\begin{align*}
\int_{[\torus_n]} f(t \rmg_n)dt \to \int_{[\G]} f(\rmg)d\rmg
\end{align*}
as $n \to \infty$ where the left-hand side denotes the integral with respect to the Haar probability measure.
\end{theorem}
\begin{remark}
If $\field_n$	is a real quadratic field for every $n$, the condition that $q_2$ splits in $\field_n$ can be omitted as the "infinite place" already splits in $\field_n$.
\end{remark}

\begin{remark}\label{remeichler}
The condition \eqref{eq:assumption compactopen} is in particular satisfied if $\compactopen_{\mathrm{f},i} = \punitsadelef{\mathcal{O}_i}$ for an Eichler order $\mathcal{O}_i$ in $\quat_i(\Q)$ -- see also \S\ref{sec:Eichlerorder} for this well-known fact.	
\end{remark}

The proof of Theorem \ref{thm:main-dynamic} is a consequence of the classification of joinings of higher rank diagonalizable actions by Einsiedler and Lindenstrauss \cite{EL-joining2} as formulated in Theorem~\ref{thm:invariance in product}. For this, the splitting assumption is crucial.
Such an assumption is commonly called a {\em Linnik-type condition} and imposing two such splitting conditions gives rise to higher-rank diagonalizable acting groups.
It is reasonable to expect that Theorems~\ref{thm:invariance in product}, \ref{thm:main-dynamic} and~\ref{thm:main-elliptic} remain true without this assumption (as is the case for a single factor). However, a proof without this assumption would require a completely different argument.
\begin{remark}In the case of two factors ($s=2$ in our notations), Valentin Blomer and Farrell Brumley have recently developed a proof of Theorem \ref{thm:main-elliptic} which is conditional on the Generalized Riemann Hypothesis but does not require any splitting assumption. As expected, the methods are entirely different from ours and build on the evaluation of fractional moments of $L$-functions in the spirit of \cite{RS,LR}.  
\end{remark}

The joinings classification \cite{EL-joining2} allows us to show that any weak${}^\ast$-limit of the measures on the packets $[\torus_n\rmg_n]$ is invariant under the image of the simply connected cover of $\G$ in $\G(\adele)$.
Thus, Theorem~\ref{thm:main-dynamic} is reduced to the verification of equidistribution for functions in the character spectrum.
Under the assumption \eqref{eq:assumption compactopen}, the character spectrum on $[\G]_{\compactopenf}$ is trivial and hence equidistribution holds. 
In Corollary \ref{cor:main theorem for general cptopen}, we provide an equivalent condition to equidistribution  for  a general compact open subgroup $\Kf$ under the  assumption that the underlying quadratic fields $\field_n$ for $\ n\geq 1$ do not belong to a finite set of exceptional fields determined by $\Kf$. 

\begin{remark}
As pointed out above, Theorem \ref{thm:main-dynamic} is a rather direct application of the classification of joinings of higher rank diagonalizable actions due Einsiedler and Lindenstrauss \cite{EL-joining2}. There is by now a number of works building on this deep and powerful result, for instance \cite{AES3D,AEShigherdim,AEW-2in4,IlyaJEMS}. As far as we know, the deepest of these applications is the work of Khayutin \cite{IlyaAnnals} who made striking progress on the mixing conjecture by Venkatesh and the third named author (see \cite{Linnikthmexpander}). 
\end{remark}

\subsection{Structure of the paper}
As this article brings together two different areas of research from dynamics and number theory, the authors have put an effort into making it accessible to both of these communities. In particular, most of the material discussed in Sections \ref{sec:ComplexElliptic}, \ref{sec:CMElliptic}, \ref{sec:supersingular} will probably be considered standard by many number theorists but we hope this will be helpful to dynamicists.   
As explained above, the main theorem (Theorem~\ref{thm:main-elliptic}) follows from Theorem~\ref{thm:main-dynamic} which is an equidistribution result of adelic toral orbits. Therefore, we first need to describe the objects above in adelic terms.
In Section~\ref{sec:ComplexElliptic} we give an adelic description of the set of complex elliptic curves. 
In Section~\ref{sec:CMElliptic} we describe adelically the complex multiplication curves, the action of the Picard group on them, and the reduction map $\red_\infty$. 
In Section \ref{sec:supersingular} we describe adelically the set of the supersingular elliptic curves together with a natural measure on them, the associated reduction map and its compatibility with the Picard group action. Then, in Section \ref{sec:diagonalComp} we combine the result of the previous sections to show compatibility between the reductions maps at different primes. This enables us to derive Theorem~\ref{thm:main-elliptic} from Theorem~\ref{thm:main-dynamic} in Section~\ref{sec:Joint equi}. The proof of Theorem \ref{thm:main-dynamic} is then done in two steps: In Section \ref{sec:simplyconnectedinvariance} we show that the limiting measure of the toral orbits is invariant under the image of the simply-connected cover and in Section \ref{sec:ProofMainDyn} we show that equidistribution holds for  $\compactopenf$-invariant functions. Finally in Section \ref{sec:refinements} we discuss further possible strengthenings of our main results.

\subsection*{Acknowledgements}{We want to express our gratitude towards the referee, whose careful reading of the manuscript improved the exposition and helped us remove an inaccuracy present in a preliminary version of this article. We also thank Valentin Blomer, Farrell Brumley, Brian Conrad, Henri Darmon, Dick Gross, Jennifer-Jayne Jakob, Ilya Khayutin, Bjorn Poonen, Dinakar Ramakrishnan and Tomer Schlank for helpful discussions. Part of this work was carried out while Ph.M.~was visiting the Department of Mathematics at Caltech, while M.L.~was visiting the Einstein institute of Mathematics at the Hebrew University of Jerusalem, and while the authors were visiting the Hausdorff Research Institute during the program ``Dynamics: Topology and Numbers''.}

\section{Notations and background}\label{sec:Notations and background}

If $E/k$ is an elliptic curve defined over a field $k$ we will in general see it as defined over an algebraic closure $\ov k$. 
In particular, its ring of endomorphism $\End(E)$ is the ring for $E/\ov k$ and likewise for the set of isogenies $\Hom(E,E')$ (plus the constant zero map) between two elliptic curves $E/k$ and $E'/k$ and similarly for the set of isomorphisms. 
Also we will usually use the same notation $E$ for an elliptic curve and its isomorphism class.

By $\adele$ (resp.~$\adelef$) we mean the ring of adeles (resp.~finite adeles) over $\bQ$ and more generally given a subset $S$ of the set $\places$ of places of $\bQ$ we let $\bQ_S$ be the restricted product of $\bQ_v$ for $v \in S$ with respect to the sequence $(\Zp)_{p\in S}$.
We also write $\fiplaces$ for the set of finite places (i.e.~for the set of rational primes).

For any affine algebraic group $\rmG$ defined over $\Qq$, we denote by $\rmG(\Aa)$ (resp.~$\rmG(\adelef)$) the group of its adelic (resp.~finite adelic) points and more generally by $\rmG(\bQ_S)=\resprod_{v\in S}\rmG(\Qv)$ the group of its $\bQ_S$-points. 
In particular, if $V$ is a finite dimensional $\Qq$-vector space, we denote the sets of its $v$-adic points, $\Q_S$-points, finite adelic, and adelic points by $V(\Qq_v)=V_v=V\otimes_\Qq\Qv$, $V(\Qq_S)$, $V(\Af)$, and $V(\Aa)$ respectively.

We will use this notation also for a number field $\field$ or more generally a finite dimensional $\Qq$-algebra.
That is, the rings of $v$-adic, $\Q_S$-, finite adelic, and adelic points of $\field$ will be denoted by $\field_v=\field(\Qv)$, $\field(\Qq_S)$, $\field(\Af)$, and $\field(\Aa)$ respectively.
Furthermore, viewing the multiplicative group $\fieldt$ of $\field$ as a $\Q$-algebraic group (which we also denote by $\Res_{\field/\Q}(\Gm)$), we write $\field^\times_v=\fieldt(\Qv)$, $\fieldt(\Qq_S)$, $\fieldt(\Af)$, and $\fieldt(\Aa)$ respectively.

\subsection{Homogeneous spaces}
Equipped with the respective natural topologies, the groups $\rmG(\Qv)$,$\rmG(\Q_S)$, $\rmG(\Af)$, and $\rmG(\Aa)$ are locally compact groups. 
The group of rational points $\rmG(\Q)$ then embeds discretely in $\rmG(\Aa)$ via the diagonal embedding. 
We denote by $[\rmG]$ the quotient
\begin{align*}
[\rmG]=\lquot{\rmG(\Qq)}{\rmG(\Aa)}
\end{align*}
and if
$\compactopen<\rmG(\Aa)$ is a subgroup, we denote the associated double quotient by 
$$[\rmG]_{\compactopen}=[\rmG]/\compactopen =
\lrquot{\rmG(\Qq)}{\rmG(\Aa)}{K}.$$
For $g\in\GA$ an element and $A\subset\GA$ a subset we denote their images in $[\G]$ or $[\G]/\compactopen$ by $[g]$, $[g]_\compactopen$ and $[A]$, $[A]_\compactopen$ respectively.
If the subgroup $K$ is understood, we will sometimes drop the index $K$ from the notation above and just write $[g]$ for the class $[g]_K$.

If $\rmG$ has no non-trivial $\Q$-characters (which will always be the case in this paper), $\rmG(\Q)$ is a lattice in $\rmG(\Aa)$ \cite[Thm.~5.5]{platonov} and the groups $\rmG(\Qv)$, $\rmG(\Q_S)$, $\rmG(\Af)$, and $\rmG(\Aa)$ are unimodular. 
In general, we denote by $m_{[\rmG]}$ (or sometimes simply $dg$) a $\rmG(\Aa)$-invariant measure on the quotient $[\rmG]$ (which is unique up to scaling). 
If $\rmG(\Q)$ is a lattice in $\rmG(\Aa)$, we always normalize $m_{[\rmG]}$ to be a probability measure.
 
For $\compactopen<\rmG(\Aa)$ a compact subgroup, we write $m_{[\rmG]_\compactopen}$ for the induced measure on $[\rmG]_{\compactopen}$ which is the restriction of $m_{[\rmG]}$ to the space of $\compactopen$-invariant continuous functions on $[\rmG]$ of compact support. This measure on $[\rmG]_\compactopen$ is a probability measure whenever $m_{[\rmG]}$ is.

Whenever $\compactopenf < \rmG(\Aa)$ is a compact open subgroup, we will call
\begin{align*}
\lrquot{\rmG(\Q)}{\rmG(\adele)}{\rmG(\R)\times \compactopenf}
\end{align*}
the class set of $\compactopenf$. This is always a finite set \cite[Thm.~5.1]{platonov}.

\subsection{Lattices}\label{sec:lattices}
A lattice $L$ in a finite-dimensional $\Q$-vector space $V$ is a finitely generated $\Zz$-module containing a $\Qq$-basis of $V$. 
We denote by $\mcL(V)$ the space of all lattices in $V$. 
In the following, $q$ will always denote a (finite) prime.
A lattice $L_q \subset V_q$ is a finitely generated $\Z_q$-module  containing a $\Q_q$-basis of $V_q$ and we denote by $\mcL(V_q)$ the space of all lattices in $V_q$. 
Given $L\in\mcL(V)$ we write $L_q=L\otimes_\Zz\Zz_q \in \mcL(V_q)$ for the closure of $L$ in $V_q$ and moreover $L_S=\prod_{q\in S}L_q$ whenever $S \subset \fiplaces$ as well as $\widehat{L}=\prod_q L_q$.

\subsubsection{The local-global principle}\label{sec:local-global}
Let $L_0\subset V$ by a fixed lattice. 
The {local-global principle for lattices}  states that the map
$$L\mapsto \what L=(L_q)_q$$ is a bijection 
\begin{align*}
\mcL(V)\simeq \mcL(\what V):=\resprod_{q}\mcL(V_q).
\end{align*}
Here, we denote by $\resprod_{q}\mcL(V_q)$ the set of sequences $(L_q)_q$ of local lattices $L_q\subset V_q$ such that for all but finitely $q$ we have $L_q=(L_{0})_{q}$.
Equivalently, $\mcL(\what V)$ is the set of $\whZ$-modules in $\what V$ commensurable with $\what{L_0}$.
We remark that the definition of $\mcL(\what V)$ is independent of the choice of lattice $L_0 \subset V$.
The inverse map is given by
\begin{equation}\label{conversemap} 
\what L=(L_q)_q\mapsto V \cap \widehat{L} := \bigcap_q V\cap L_q \subset V.	
\end{equation}

\subsubsection{Interpretation in terms of adelic quotients}

Let $\GL_V$ denote the general linear group of $V$. 
If $V=\Qq^n$, we denote this group by $\GL_n$. 
The group $\GL_V(\Qq)$ acts transitively on $\mcL(V)$ and each of the local groups $\GL_V(\Q_q)$ acts transitively on the local space $\mcL(V_q)$.
These local actions induce a transitive action
$$\GL_V(\Af)\curvearrowright \mcL(\what V)$$ 
and therefore a transitive action $\GL_V(\Af)\curvearrowright \mcL(V)$ by the local-global principle: given $g_\mathrm{f}\in\GL_V(\Af)$ and $L$ a lattice in $V$ we have
$$g_\mathrm{f}.L=V\cap(g_\mathrm{f}.\what L).$$ 
This action is compatible with the classical action $\GL_V(\Qq)\lact\mcL(V)$: if $\delta_\mathrm{f}:\GL_V(\Qq)\hookrightarrow \GL_V(\Af)$ denotes the diagonal embedding, one has for any lattice $L$ in $V$ and any $g_\Q \in \GL_V(\Q)$
$$g_\Qq.L=\delta_\mathrm{f}(g_\Qq).L.$$
Let $\GL_V(\Zz)$, $\GL_V(\whZ)$ be the stablilizers of $L_0$ under the corresponding actions.
It follows that
\begin{align}\label{localglobalquotient}
\mcL(V)
&\simeq \rquot{\GL_V(\Qq)}{\GL_V(\Zz)}
\simeq \rquot{\GL_V(\Af)}{\GL_V(\whZ)}\\
&\simeq \lrquot{\GL_V(\Qq)}{\GL_V(\Qq)\times\GL_V(\Af)}{\GL_V(\whZ)}.\nonumber
\end{align}
For the last bijection the group $\GL_V(\Qq)$ is embedded diagonally in the product $\GL_V(\Qq)\times\GL_V(\Af)$ and the bijection is induced by the map
$$g_{\Qq}g_\mathrm{f}\in\GL_V(\Qq)\times\GL_V(\Af)\mapsto g_{\Qq}^{-1}g_\mathrm{f}.L_0\in\mcL(V).$$

Let $\mcL(V_\infty)$ be the space of `real' lattices in $V_\infty=V\otimes_\Qq\Rr$ (i.e.~discrete subgroups of maximal rank).  
We have the identification
$$\mcL(V_\infty)\simeq\GL_V(\Rr)/\GL_V(\Zz)$$
and the following identification in terms of adelic groups : 
\begin{equation}\label{localglobalquotientreal}
\mcL(V_\infty) \simeq \GL_V(\Qq)\bash \GL_V(\Aa)/\GL_V(\whZ).	
\end{equation}
In the latter, $\GL_V(\Qq)$ is embedded diagonally in $\GL_V(\Aa)$ and the bijection  is induced by the map
\begin{equation}\label{adeliclatticemap}
g_{\infty}g_\mathrm{f}\in\GL_V(\Aa)\mapsto g_{\infty}^{-1}g_\mathrm{f}.L_0\in\mcL(V_\infty).	
\end{equation}

\subsubsection{Commensurability}\label{seccommens} 
We recall the following definition.

\begin{definition}
 Two lattices $L,L'\in\mcL(V_\infty)$ are commensurable if there is an integer $n\not=0$ such that $nL'\subset L$.	
\end{definition}

For any $L \in \mcL(V_{\infty})$ we denote by
\begin{align*}
L^0:=L\otimes_\Zz\Qq
\end{align*}
the $\Q$-subspace of $V_\infty$ spanned by $L$.
Then two lattices $L,L'\in\mcL(V_\infty)$ are commensurable if and only if $L^0=(L')^{0}$.
Hence, the commensurability class of $L$ is simply the set of lattices in $L^0$, that is, 
$$\mcL(L^0)=\GL_V(\Qq).L.$$
Here, $\GL_V(\Qq)$ is viewed as a subgroup of $\GL_V(\Rr)$. Alternatively, if $g_\infty\in\GL_V(\Rr)$ is such that $g_\infty^{-1}L_0=L$, then  
$$\mcL(L^0)=g_\infty^{-1}\GL_V(\Qq).L_0=g_\infty^{-1}\GL_V(\Af).L_0.$$
In other terms the map \eqref{adeliclatticemap} provides the following identitifation for the commensurability class of $L$:
\begin{equation}\label{adeliccomens}
\mcL(L^0)\simeq\GL_V(\Qq)\bash \GL_V(\Qq)g_{\infty}\GL_V(\Af)/\GL_V(\whZ).	
\end{equation}
By the set on the right-hand side we mean the image under the quotient map to \eqref{localglobalquotientreal} of the set $\GL_V(\Qq)\bash \GL_V(\Qq)g_{\infty}\GL_V(\Af)$.
Note that the copies of $\GL_V(\Q)$ in \eqref{adeliccomens} are embedded differently in $\GL_V(\Aa)$, the left copy is embedded diagonally while the right copy is a priori embedded in $\GL_V(\R)$.
However, the presence of $\GL_V(\Af)$ allows us to take $\GL_V(\Q)$ diagonally embedded in either case.

\begin{remark} This is only interesting for the place $v=\infty$; the lattices in $\mcL(V)$ or $\mcL(V_p)$ are all commensurable.
\end{remark}

\subsection{Quadratic fields and tori}

In this paper, $\field$ will usually denote a quadratic extension of $\Qq$ (in Sections~\ref{sec:CMElliptic}--\ref{sec:Joint equi} always imaginary).
We denote by $\Tr_{\field}$ and $\Norm_{\field}$ the trace and norm on $\field$ and denote in the same way the natural extensions to $\field_v$, $\field(\Q_S)$, $\KAf$ and $\KA$.
Recall that we view the multiplicative group $\fieldt$ of $\field$ as a $\Q$-algebraic group (which is $\res_{\field/\Qq}\Gm$). 
We will usually denote the projective group by
$$\torus_\field=\Res_{\field/\Qq}(\Gm)/\Gm.$$

\subsubsection{Quadratic Orders}\label{sec:quadraticorders}
We write $\order_\field$ for the ring of integer of $\field$.
Recall that a subring $\order\subset\field$ is an order if $\order$ is a $\Zz$-lattice in $\field$.
It follows that $\order\subset\order_K$ and that there is a unique integer $c=c(\order)>0$ (the {\em conductor of }$\order$) such that
$$\order=\Zz+c\order_\field.$$
The discriminant of $\order$ is $\disc(\order)=\det (\Tr_{\field}(e_ie_j))_{i,j=1,2}$ where $(e_1,e_2)$ is any $\Zz$-basis of $\order$. One has
$$\disc(\order)=\disc(\order_\field)c^2.$$

\subsection{Quaternion Algebras}
In this paper $\quat$ will usually denote a quaternion algebra over $\Qq$ (possibly the {\em split }quaternion algebra $\Mat_2$ of $2\times 2$ matrices over $\Qq$). 
Given a (possibly empty) finite set of places $S$ of $\bQ$ with $\av{S}$ even, there is up to $\bQ$-isomorphism a unique quaternion algebra $\quat_S$ ramified exactly at the places in $S$ (see e.g.~\cite[Thm.~14.6.1]{voight}).
Any quaternion algebra over $\Q$ is of this form.
We say that $\quat$ is \emph{definite} if $\quat(\R)$ is isomorphic to the algebra of Hamiltonian quaternions and otherwise we say that it is indefinite. 

Denote by $\Tr_{\quat}$ and $\Norm_{\quat}$ the (reduced) trace and (reduced) norm on $\quat$. 
Let $\quat^0\subset\quat$ be the subspace the quaternions of trace $0$ and let $\quat^1$ be the group the quaternions of norm $1$.

The $\Q$-algebraic group of units $\quat^\times$ acts on the $\Qq$-space $\quat^{0}$ by conjugation. 
The image of $\quat^\times$ in $\GL_{\quat^{0}}$ is isomorphic to the quotient of $\quatt$ by its center and is called the projective group of units of $\quat$. 
We denote it by 
\begin{align*}
\Gm\bash \quatt=\PBt.
\end{align*}
The group $\PBt$ is $\bQ$-almost simple with trivial center and its simply connected cover is the group of norm one units $\quat^1$. These two connected groups are anisotropic over $\bQ$ if and only if $\quat \neq \Mat_2$. 

Notice that the conjugation action preserves the norm form $\Norm_{\quat}$ restricted to $\quat^{0}$, therefore $\PBt\subset \SO_{\Norm_\quat|_{\quat^{0}}}$, and since  these two groups are connected, we have an isomorphism of $\bQ$-groups $$\PBt \cong \SO_{\Norm_\quat|_{\quat^{0}}},$$ see  \cite[Ch.~I, Thm.~3.3]{vigneras}.

Write $\G = \PBt$ for simplicity.
By the above, the homogeneous space $[\G]$ for the group $\G$ has finite volume (i.e.~$\G(\Q)$ is a lattice in $\G(\Aa)$).
We note that if $\quat \neq \Mat_2$, $[\G]$ is in fact compact.

%

\subsubsection{Integral structures}\label{sec:integral_structures}
The choice of an order $\mathcal{O}\subset \quat(\bQ)$ defines an integral structure on $\G = \PBt$ by setting
\begin{align*}
\G(\Z) = \{g \in \G(\bQ): g\mathcal{O}g^{-1} = \mathcal{O}\}
\end{align*}
The integral structure also provides  distinguished local and global compact open subgroups, namely, for any prime $p$ the subgroups $$\G(\Z_p)=\{g \in \G(\Qp): g\mathcal{O}_p g^{-1} = \mathcal{O}_p\}$$ and the product 
$$\G(\widehat{\Z})=\prod_p\G(\Z_p)<\G(\Af).$$
 If $\quat=\Mat_{2}$, we will always choose the integral structure defined by $\mcO=\Mat_{2}(\bZ)$. In this case, $\G(\Z)$ equals $\PGL_{2}(\bZ)$, i.e.~the image of $\GL_{2}(\Z)$ in $\PGL_{2}(\bQ)$.

\subsubsection{Quadratic fields and Quaternions}
Let us recall that a quadratic number field $\field$ embeds in a quaternion algebra $\quat$ if and only if the following local conditions are satisfied:
  \begin{enumerate}[(i)]
    \item If $\quat$ is definite, then $\field$ is imaginary.
    \item If $p$ is a ramified prime of $\quat$, then $p$ is non-split in $\field$.
  \end{enumerate}
This is a consequence of the Hasse-Minkowksi theorem applied to the quadratic form $\Norm_\quat|_{\quat^{0}}$.

Suppose we have such an embedding 
$$\iota:\field\hookrightarrow \quat.$$
Then we have for any $z\in\field$
$$\Tr_\quat(\iota(z))=\Tr_\field(z),\ \Nr_\quat(\iota(z))=\Nr_\field(z).$$
The embedding $\iota$ also induces an embedding for the multiplicative and projective $\Qq$-groups which we denote in the same way by
$$\iota:\fieldt\hookrightarrow \quatt,\ \torus_\field\hookrightarrow \G = \PBt.$$
In fact $\iota(\fieldt)\subset\quatt$ and $\iota(\torus_\field)\subset \PBt$ are the stabilizers in $\quatt$ or $\PBt$ of the line $\Qq.\iota(z_0)\subset \quat^0(\Qq)$ where $z_0\in\field$ is any non-zero element with trace~$0$.

\section{Complex elliptic curves} \label{sec:ComplexElliptic}
Let $\EllC$ be the set of elliptic curves over $\Cc$ up to isomorphism. 
Any elliptic curve $E$ admits a complex uniformization, that is, it is isomorphic as a Riemann surface to a quotient
\begin{equation*}
E\simeq\Cc/\Lambda
\end{equation*}
where $\Lambda\subset \Cc$ is a $\Zz$-lattice.
Conversely, any lattice gives rise to an elliptic curve over $\bC$; cf.~\cite[Chap.~7, \S2]{serre}.
 
Given two elliptic curves $E\simeq \Cc/\Lambda$, $E'\simeq \Cc/\Lambda'$, the set of isogenies from $E$ to $E'$ is given by
$$\Hom(E,E')\simeq \{z\in\Cc: z\Lambda\subset \Lambda'\}.$$ 
In particular,  two curves are isomorphic if and only if $\Lambda'=z\Lambda$ for some $z\in\Ct$ (i.e.~the lattices $\Lambda$ and $\Lambda'$ are $\Ct$-homothetic).
Taking $E'=E$ one has
\begin{equation*}
\End(E)=\{z\in\Cc: z\Lambda\subset\Lambda\}\hbox{ and }\Aut(E)=\End(E)^\times.
\end{equation*}
The endomorphism ring $\End(E)$ is either $\Zz$ (the generic case) or an order $\order$ in an imaginary quadratic field $\field$. In this last case, one says that $E$ has (exact) complex multiplication by $\order$ and we discuss this further in the next section. In any case we set
\begin{align*}
&\End^0(E)=\End(E)\otimes_\Zz\Qq\subset\Cc \text{ and} \\
&\Aut^0(E)=\End^0(E)^\times = \End^0(E) \setminus\{0\}\subset\Ct.
\end{align*}

Let $\mcL(\Cc)$ be the space of $\Z$-lattices in $\Cc$. The map 
\begin{equation}\label{latticeEll}
\Lambda\in\mcL(\Cc)\mapsto E_\Lambda=\Cc/\Lambda
\end{equation}
  induces a bijection 
$$\Ct\bash \mcL(\Cc)\simeq \EllC.$$
Furthermore, any $\Z$-lattice in $\C$ is $\Ct$-homothetic to a lattice of the form $$\Lambda_z=\Zz+\Zz\,z$$
for some $z\in\Hh$ 
where $\Hh=\{z\in\Cc: \Im z>0\}$ is the upper-half plane.
Moreover, two such lattices $\Lambda_z, \Lambda_{z'}$ are homothetic if and only if $z,z'$ are in the same $\SL_2(\Z)$-orbit.

From this we obtain the usual identification
$$\EllC\simeq \SL_2(\Zz)\bash \Hh=Y_0(1)$$
of $\EllC$ with the complex points of the modular curve.

 In the sequel,  we identify $\Cc$ with the Euclidean plane $\Rr^2$ by matching the $\Rr$-basis $(1,\ii)$ of $\Cc$ with the canonical basis $((1,0),(0,1))$ of $\Rr^2$.
 We denote this identification by
 \begin{align}\label{eq:C = R^2}
 \theta_\ii:\Cc\simeq \Rr^2.
 \end{align}
Under this identification, the action of the group $\Ct\simeq\Rr_{>0}\times\mathbb{S}^{1}$ on $\C$ by multiplication corresponds to the standard action of $\R_{>0}\times \SO_2(\R)$ on $\R^2$. 
By \eqref{localglobalquotientreal} we have the adelic description
\begin{align}\nonumber
\Ct\bash\mcL(\Cc)&\simeq \Rr_{>0} \SO_2(\Rr)\bash\mcL(\Rr^2)\\
&\simeq \GL_2(\Qq)\bash \GL_2(\Aa)/\Rr_{>0}\SO_2(\Rr)\GL_2(\whZ).\label{EllCadelic}
\end{align}
Since $\mathbb{A}^\times = \Q^\times \R_{>0}\widehat{\Z}^\times$, $\Ct\bash\mcL(\Cc)$ is  expressed as an adelic quotient of the projective group
\begin{align*}
\Ct\bash\mcL(\Cc)&\simeq  \PSO_2(\Rr)\bash \PGL_2(\Rr)/\PGL_2(\Zz)\\
&\simeq \PGL_2(\Qq)\bash \PGL_2(\Aa)/\PSO_2(\Rr)\PGL_2(\whZ).\end{align*}
Recall here from \S\ref{sec:integral_structures} that $\PGL_2(\whZ)$ denotes the image of $\GL_2(\whZ)$ in $\PGL_2(\Af)$ under the natural projection. We denote the resulting identification
\begin{equation}\psi_\infty:\Ell_\infty\simeq \PGL_2(\Qq)\bash \PGL_2(\Aa)/\PSO_2(\Rr)\PGL_2(\whZ)\label{identcomplexadelic}	
\end{equation}

\subsection{Isogenies}\label{secisogeny}
Given $E=\Cc/\Lambda$ a complex  elliptic curve, let 
\begin{align*}
Y(E)=\{[E']: \Hom(E',E)\not=0\}\subset\EllC
\end{align*}
be the set of isomorphism classes of complex elliptic curves $E'$ isogenous to $E$ (recall that being isogenous is an equivalence relation). 

Given $E'\simeq\Cc/\Lambda'$ isogenous to $E$, there is $z\in\Ct$ such that $z\Lambda'\subset\Lambda$. 
In other terms, $z\Lambda'$ is commensurable with $\Lambda$; cf.~\S \ref{seccommens}. 
Conversely, if $\Lambda'$ is commensurable with $\Lambda$ then $E'\simeq \Cc/\Lambda'$ is isogenous to $E$: given $n\not=0$ such that $n\Lambda'\subset\Lambda$ the map
$$[\times n]:z+\Lambda'\in\Cc/\Lambda'\mapsto nz+n\Lambda'+\Lambda\in\Cc/\Lambda$$
gives an isogeny. Therefore, the map \eqref{latticeEll} gives an identification between $Y(E)$ and $\C^\times$-homothety classes of lattices commensurable to $\Lambda$. 
In other words, it induces a bijection
$$Y(E)\simeq \Ct\bash\Ct\mcL(\Lambda^0)$$
Here, $\Lambda^0=\Lambda\otimes_\Zz\Qq$, see \S \ref{seccommens} for the notation. 
Let $g_\infty\in\GL_2(\Rr)$ be such that $\Lambda=g_\infty^{-1}.\Zz^2$.
In terms of the identifications \eqref{EllCadelic} and \eqref{identcomplexadelic} we have (cf.~\eqref{adeliccomens})
 \begin{equation}\label{isogenyadelicGL}
 Y(E)\simeq \GL_2(\Qq)\bash \GL_2(\Qq)g_\infty\GL_2(\Af)/\Rr_{>0}\SO_2(\Rr)\GL_2(\whZ).	
 \end{equation}
and letting $\ov g_\infty$ denote the projection of $g_\infty$ to $\PGL_2(\Rr)$) we obtain
 \begin{equation}\label{isogenyadelic}
 \psi_{\infty}|_{Y(E)}:
 Y(E)\simeq \PGL_2(\Qq)\bash \PGL_2(\Qq)\ov g_\infty\PGL_2(\Af)/\PSO_2(\Rr)\PGL_2(\whZ).
 \end{equation}
 
\begin{remark}
Since the stabilizer of $\Lambda^0$ in $\Ct$ is $\Aut^0(E)$, we also have
\begin{equation}\label{YEbij}
Y(E)\simeq \Aut^0(E)\bash\mcL(\Lambda^0).	
\end{equation}
\end{remark}

\section{CM elliptic curves}
\label{sec:CMElliptic}

\subsection{Curves with complex multiplication}

Let us recall that given an imaginary quadratic field $\field\subset\Cc$, a complex elliptic curve $E\simeq\Cc/\Lambda$  has complex multiplication (CM) by $\field$ if
\begin{equation}
\End^0(E):=\End(E)\otimes_\Zz\Qq= \field\subset \Cc.
\end{equation} 
In that case the ring of endomorphisms $\End(E)$ is isomorphic to an order $\orderfield$ in $\field$:
\begin{equation}
\End(E)\simeq \ordernumberfield
\end{equation}
and one then says that $E$ has CM by $\orderfield$.

We denote by $\CM_\field$  the set of $\C$-isomorphism classes of elliptic curves $E$ with complex multiplication by $\field$ 
and by $\CM_D$ or $\CM_\ordernumberfield\subset\CM_\field$ the subset consisting of elliptic curves $E$ with CM by $\ordernumberfield$ where $D=\disc(\order)$. 
We have a decomposition as a disjoint union 
$$\CM_\field=\bigsqcup_{\ordernumberfield\subset\field}\CM_{\ordernumberfield}$$
where $\ordernumberfield$ ranges over all the orders of $\field$.

\subsubsection{Explicit description in terms of the complex uniformization}
As a general remark, an elliptic curve $E$ given by its complex uniformization $E(\Cc)=\Cc/\Lambda$
has CM by an order $\order\subset\field$ if and only if
\begin{equation*}
\End(\Lambda)=\{z\in\Cc:\ z\Lambda\subset\Lambda\}= \order.	
\end{equation*}
An example of such an elliptic curve is given by $E_\order(\Cc)=\C/\ordernumberfield$.

\begin{proposition}\label{propCMisog} 
The map \eqref{latticeEll} induces a bijection 
$$\CM_\field\simeq \Kt\bash \mcL(\field)$$
where $\mcL(\field)$ denotes the set of $\Zz$-lattices in the $\Qq$-vector space $\field$ as in \S \ref{sec:lattices}.
In particular, all elliptic curves with CM by $\field$ are isogenous and for any $E\in\CM_\field$ we have
$$\CM_\field=Y(E).$$
\end{proposition}

\begin{proof}
For any elliptic curve $E \in \CM_{\field}$ with $E = \bC/\Lambda$ the lattice $\Lambda$ is homothetic to a lattice containing an order $\ordernumberfield$ in $\field$. In particular, the latter is commensurable to $\ordernumberfield$.
It follows from Section~\ref{secisogeny} that $E$ is isogenous to $E_\ordernumberfield$.
Conversely, if a lattice $\Lambda$ is homothetic to a lattice in $\field$, it has complex multiplication by an order in $\field$.
\end{proof}

\subsubsection{The Picard group and its action} We recall the structure of $\CM_\order$ for $\order\subset\field$ an order.

\begin{definition}\label{defproperorder} 
A (fractional) proper $\order$-ideal is a lattice $\mfa\subset \field$ such that 
$$\End(\mfa)=\{z\in\field: z\mfa\subset\mfa\}= \ordernumberfield.$$
We denote  the set of proper $\order$-ideals by $\mcI(\order)\subset \mcL(\field)$.
\end{definition}

The group $\fieldt$ acts on $\mcI(\order)$ by multiplication and from Definition~\ref{defproperorder} and the proof of Proposition~\ref{propCMisog}, the restriction of the map \eqref{latticeEll} to $\mcI(\order)$ yields the bijection
\begin{equation}\label{bijCMorder}
\CM_\order\simeq \fieldt\bash \mcI(\order).	
\end{equation}
 
Let us recall that $\order$ is a Gorenstein ring, so the set of proper $\order$-ideals $\mcI(\order)$ is exactly the set of $\order$-modules in $\field$ which are locally free of rank $1$  and so invertible \cite[\S 2.3]{ELMV-Ens}. 
The set $\mcI(\order)$ is therefore stable under multiplication of $\order$-ideals and forms a commutative group whose neutral element is  $\order$. 

\begin{definition}\label{defpicard} 
The quotient $\Pic(\order)=\Kt\bash \mcI(\order)$ is called the {\em Picard group} of $\order$.
\end{definition} 

It is a very classical result that $\Pic(\order)$ is a finite group, see for example \cite[Thm.~2.13, 7.7]{cox}.
The set $\CM_\order$ is in bijection with $\Pic(\order)$; it is not a group a priori but comes very close being one.

 \begin{proposition}\label{propPicaction} 
The set $\CM_\orderfield$ is endowed with a simply transitive action of the Picard group $\Pic(\orderfield)$.	
 \end{proposition}

\begin{proof}
The action of the Picard group is defined as follows. 
Let $E \in \CM_\order$ and by \eqref{bijCMorder} let $\Lambda \in \mcI(\order)$ such that $E(\C) \simeq \C/\Lambda$.
Furthermore, let $\mfa$ be a proper $\order$-ideal and write $\mfa^{-1} \in \mcI(\order)$ for its inverse.
The set $\mfa^{-1}\Lambda$ is in $\mcI(\order)$ and hence is a lattice in $\C$ satisfying 
$\End(\mfa^{-1}\Lambda)=\ordernumberfield$. 
It follows that
\begin{equation}\label{mfastardef}
\mfa\star E:=\Cc/\mfa^{-1}\Lambda	
\end{equation}
has CM by $\order$ and it is easy to check that its isomorphism class depends only on the classes of $\mfa$ and $E$. This action is transitive since
$\mfa\star E_\order= \C/\mfa^{-1}$
and simply transitive by \eqref{bijCMorder}.
\end{proof}

\begin{remark}
For the sequel it will be useful to take note of the isomorphism of $\order$-modules 
$$\Hom(\mfa\star E,E)=\{z\in\Cc:\ z\mfa^{-1}\Lambda\subset\Lambda\}=\mfa$$
where $\order$ acts on the left on $\Hom(\mfa\star E,E)$ via postcomposition by isogenies of $E$. 
Conversely, one verifies that the map $u+\mfa^{-1}\Lambda\mapsto \bigl[z\in\mfa\ra zu+\Lambda\bigr]$
induces a bijection
\begin{equation}\label{serre1}
	\Hom_\order(\mfa,E)\simeq \mfa\star E(\Cc)
\end{equation}
where $\Hom_\order$ denotes the $\order$-module homomorphisms.
\end{remark}

 \subsubsection{The Galois action}
Any elliptic curve $E$ with CM by $\order$ has a model defined over the ring class field $\ringclassfield{\ordernumberfield}$ and the $j$-invariant satisfies
 $\field(j(E))=\ringclassfield{\ordernumberfield}.$
In addition, the Galois and the Picard actions are compatible
in the sense that
 \begin{equation*}
 [\mfa\star E]=[E^{\sigma}]
 \end{equation*}
 where $[\mfa]\in\Pic(\OK)$ and $\sigma=(\frac{\ringclassfield{\ordernumberfield}/\field}{[\mfa]})$ is the Artin symbol from class field theory.
%
 See \cite[p.~293]{SerreCM} for a proof when $\order=\order_\field$ is the full ring of integers and \cite[\S 22]{sutherland} in general.
%

\subsection{The set of CM elliptic curves as an adelic quotient}\label{sec:CM as adelic quot}
We now give an adelic description of the action
$$\Pic(\order)\lact \CM_\order\subset \CM_\field\subset\Ell_\infty.$$
Let $\GL_\field=\End_\Qq(\field)^\times$ be the linear group of the $\Qq$-vector space $\field$. We have an injection $\fieldt\hookrightarrow \GL_\field(\Qq)$ via the multiplicative action of $\fieldt$ on $\field$. By Proposition \ref{propCMisog}  and the local-global principle for lattices, we have
\begin{align*}\label{adelicCMfield}
\CM_\field &\simeq \fieldt \bash \mcL(\field)\simeq \fieldt \bash \GL_{\field}(\Af).\what\order\\
&\simeq \fieldt\bash \GL_{\field}(\Af)/\GL_\field(\whZ).
\end{align*} 
Here $\GL_\field(\whZ)$ is the stabilizer of $\whorder$ under the action of $\GL_\field(\Af)$ on $\KAf$.

By Proposition \ref{propPicaction} we have that the inclusion
$$\CM_\order\subset \CM_\field$$ corresponds to the homothety classes of lattices which are proper $\order$-ideals:
 $$\fieldt\bash\mcI(\order)\subset\fieldt\bash\mcL(\field).$$
  As already discussed the proper $\order$-ideals are precisely the locally free ones: for any $\mfa\in\mcI(\order)$ there is $t_\mathrm{f}\in\KtAf$ such that
\begin{equation}\label{aidele}
\mfa=\field\cap\what\mfa\hbox{ where }\what\mfa= t_\mathrm{f}\what\order.
\end{equation}
Conversely, for any $\what\mfa = t_\rmf \what\order$ the formula \eqref{aidele} defines a proper $\order$-ideal.
Moreover, $t_\mathrm{f}$ is uniquely defined modulo $\what\order^\times$ and therefore
$$\fieldt\bash\mcI(\order)=\Pic(\order)\simeq \fieldt\bash\fieldt(\Af)/\whOrt$$
which gives an adelic description of the group $\Pic(\order)$.

If one prefers to think in terms of homothety classes of lattices one has
$$\fieldt\bash\mcI(\order)\simeq\fieldt\bash\fieldt(\Af).\whorder\simeq \fieldt\bash\fieldt(\Af). \GL_\field(\whZ)/ \GL_\field(\whZ).$$

Indeed, note that by definition of $\GL_\field(\whZ)$ the stabilizer of $\widehat{\order}$ is the stabilizer of the identity coset in $\rquot{\GL_\field(\adelef)}{\GL_\field(\whZ)}$. It is equal to  
	$\KtAf\cap\GL_\field(\whZ)=\whOrt$.
	Thus, the two descriptions are compatible:
	$$\fieldt\bash \fieldt(\Af). \GL_\field(\whZ)/ \GL_\field(\whZ)\simeq \fieldt\bash\fieldt(\Af)/\whOrt.$$	

To summarize, we have the identifications
\begin{align}\label{Picidele1}\CM_\order& \simeq\fieldt\bash \KtAf/\whOrt \simeq \Pic(\order)\\
&\simeq \fieldt\bash\fieldt(\Af). \GL_\field(\whZ)/ \GL_\field(\whZ)\label{Picidele2}
\end{align}
and the action of $\Pic(\order)$ on $\CM_\order$ is realized as follows: 

\begin{proposition}\label{propidentinfty}
	 If  $\mfa\in\mcI(\order)$ and $t_\mathrm{f}\in\fieldt(\Af)$ are related by \eqref{aidele}, the isomorphism class of 
$\mfa\star E_\order$ is represented in \eqref{Picidele1} resp.~\eqref{Picidele2} by the class
$$[t_\mathrm{f}^{-1}]=\fieldt t_\mathrm{f}^{-1}\whOrt
\quad \text{resp.} \quad
\fieldt t_\mathrm{f}^{-1}.\GL_\field(\whZ).$$
\end{proposition}

\begin{proof}
From the proof of Proposition~\ref{propPicaction} it follows that the homothety class of lattices corresponding to $\mfa\star E_\order$ under the identification $\Ct\bash \mcL(\Cc)\simeq \EllC$ is given by the lattice $\mfa^{-1} \in\mathcal{I}(\order) \subset \mathcal{L}(\field)$.
The statement is thus a direct consequence of \eqref{aidele}.
\end{proof}

\subsection{The map $\red_\infty$}\label{sec:red_infty}
In this section we describe the map $\red_\infty$ adelically. Let us recall that
\begin{align*}
\Ell_\infty &\simeq \GL_2(\Qq)\bash\GL_2(\Aa)/\Rr_{>0}\SO_2(\Rr)\GL_2(\whZ)\\
&\simeq \PGL_2(\Qq)\bash \PGL_2(\Aa)/\PSO_2(\Rr)\PGL_2(\whZ)	
\end{align*}
As we have seen, the elliptic curves with CM by $\field$ 
 correspond to the isogeny class of $E_\order$ and to the sublattices of $\field$
 \begin{equation}\label{identificationCMandLattices}
 \CM_\field=Y(E_\order)\simeq \Ct\bash \Ct\mcL(\field)
 \end{equation}
  where $\field$ is viewed as a $\Qq$-vector space of dimension $2$ inside the $\Rr$-vector space $\C = \R^2$ and the sublattices of $\field$ are the images of one of them (say $\order$) under all the invertible $\Qq$-linear maps on that space (see \S\ref{seccommens}). Concretely, we choose a $\Zz$-basis $\order=\Zz\omega_1+\Zz\omega_2$.
  This choice gives isomorphisms
 \begin{equation}\label{thetadef}
 \theta:\order\to \Zz^2,\quad \theta:\field\to\Qq^2	
 \end{equation}
 as well as  $\Zz$ and $\Qq$-algebra embeddings 
$$\iota:\order=\End(\order)\hookrightarrow \mcO:=\Mat_2(\Zz),\ \iota:\field=\End(\order)^0\hookrightarrow \quat:=\Mat_2(\Qq)$$
satisfying
\begin{itemize}
\item Compatibility: for all $z\in\fieldt$ and  $w\in\field$
$$\theta(z.w)=\iota(z)(\theta(w)).$$
\item Optimality:
 \begin{equation}\label{optimapeq}
 \iota(\field)\cap \mcO=\iota(\ordernumberfield).
 \end{equation}
\end{itemize}
Restricting $\iota$ to $\fieldt$, we obtain an embedding of $\Qq$-algebraic groups
$$\iota:\Res_{\field/\Qq}(\Gm)\hookrightarrow \GL_{2}$$
where we recall that we view $\fieldt = \Res_{\field/\Qq}(\Gm)$.
We denote the image by $\iota(\field^\times)$ for which \eqref{optimapeq} translates into 
$$\iota(\field^\times)(\Af)\cap\GL_2(\whZ)=\iota(\whOrt).$$

\begin{remark} 
If one chooses another basis, the induced embedding $\iota'$ will be $\GL_2(\Q)$-conjugate to $\iota$ and in particular the corresponding torus $\iota'(\Res_{\field/\Qq}(\Gm))$ will be $\GL_2(\Qq)$-conjugate to $\iota(\Res_{\field/\Qq}(\Gm))$.
\end{remark}

Let $g_\infty\in\GL_2(\Rr)$ be the $\Rr$-linear extension of $\theta$ where we have $\field \otimes_{\Q} \R = \C = \R^2$ under \eqref{eq:C = R^2}. 
We have
\begin{equation}\label{conjiota}
\iota(\field^\times)(\Rr)=g_\infty\Rr_{>0}\SO_2(\Rr)g_\infty^{-1}
\end{equation}
as well as $\field = g_{\infty}\Q^2$.
As of \eqref{identificationCMandLattices} we obtain by applying the characterization of commensurability classes in \eqref{adeliccomens} (see also \eqref{isogenyadelicGL})
\begin{equation}\label{CMadele1}
\CM_\field\simeq \GL_2(\Qq)\bash\GL_2(\Qq)g_\infty\GL_2(\Af)/\Rr_{>0}\SO_2(\Rr)\GL_2(\whZ)
\end{equation}
by identifying elliptic curves over $\C$ with homothety classes of lattices $\C$.
The set $\CM_\order$ correspond under this identification to the homothety classes of the lattices $$g^{-1}_\infty.\iota(\fieldt)(\Af).\Zz^2$$
by the adelic description in \S\ref{sec:CM as adelic quot}.
As a subset of \eqref{CMadele1}, $\CM_\order$ corresponds to 
\begin{gather*}
\GL_2(\Qq)g_\infty\iota(\fieldt)(\Af)\Rr_{>0}\SO_2(\Rr)\GL_2(\whZ)	\\
=\GL_2(\Qq)\iota(\fieldt)(\Aa)g_\infty\Rr_{>0}\SO_2(\Rr)\GL_2(\whZ)
\end{gather*}
by \eqref{conjiota}. 
More precisely if $\mfa\in\mcI(\order)$ and $t_\mathrm{f}\in\fieldt(\Af)$ are related by \eqref{aidele}, the homothety class of the lattice $\mfa$ is given by 
$$[\iota(t_\mathrm{f})g_\infty]=\GL_2(\Qq)\iota(t_\mathrm{f})g_\infty\Rr_{>0}\SO_2(\Rr)\GL_2(\whZ).$$

Passing to the projective group, let
$$\iota:\torus_\field = \Res_{\field/\Qq}(\Gm)/\Gm\hookrightarrow \PGL_2$$ be the induced embedding. 
For $t \in \fieldt$ we denote by $\overline{t} \in \torus_\field$ the image.

From \eqref{identcomplexadelic} recall that we denote by $\psi_\infty$ the identification of $\EllC$ with an adelic double quotient of $\PGL_2$.
Summarizing the above, we get the following proposition.

\begin{proposition}\label{propredinfty}Let $\ov g_\infty\in\PGL_2(\Rr)$ be the image of $g_\infty$ under the natural projection. 
Under $\psi_\infty$ the set $\CM_\order$is mapped to
$$\psi_\infty(\CM_\order)
=[\iota(\torus_\field)(\Aa).\ov g_\infty]
=\PGL_2(\Qq)\iota(\torus_\field)(\Aa)\ov g_\infty\PSO_2(\Rr)\PGL_2(\whZ)$$
which is the image of the adelic torus orbit $\PGL_2(\Qq)\iota(\torus_\field)(\Aa)\ov g_\infty$.
Moreover, if $\mfa\in\mcI(\order)$ and $t_\mathrm{f}\in\fieldt(\Af)$ are related by \eqref{aidele}, i.e.~$\what\mfa = t_\mathrm{f}.\what\order$, then
$\red_\infty(\mfa\star E_\order)$ is represented by the class 
$$\psi_\infty(\red_\infty(\mfa\star E_\order))
=[\iota(\ov t_\mathrm{f}^{-1})\ov g_\infty]=\PGL_2(\Qq)\iota(\ov t_\mathrm{f}^{-1})\ov g_\infty\PSO_2(\Rr)\PGL_2(\whZ).$$
\end{proposition}

\section{Supersingular elliptic curves}\label{sec:supersingular}

Recall that an elliptic curve $E_0/\ov k$ defined over an algebraically closed field is \emph{supersingular} if its endomorphism ring $\End(E_0)$ is an order in a quaternion algebra.
In that case, the characteristic of $\ov k$ is a prime, say $p$ and the quaternion algebra 
\begin{equation}\label{Endquat}
\quat=\End(E_0)\otimes \bQ 
\end{equation}
is (isomorphic to) the quaternion algebra $\quat_{\infty,p}$ defined over $\Qq$ ramified exactly at $\infty$ and $p$.
Moreover, $$\mcO:=\End(E_0)\subset \quat$$ is a {\em maximal} order; cf.~\cite{deuring41} or~\cite[Thm.~42.1.9]{voight}. In addition $E_0$ has a model defined over $\Ff_{p^2}$.
In particular, the set of $j$-invariants of such supersingular curves over $\overline{k}$ is finite and so there are only finitely many isomorphism classes.
We denote by $\SS_p$ the (finite) set of isomorphism classes of supersingular elliptic curves over $\overline{\Fp}$.

The first aim of this section is to recall how $\SS_p$ is realized as an adelic quotient. We will show the following.

\begin{proposition}\label{prop:Bpinfty}
Let $\G=\PBt = \quatt/\Gm$ be the projective group of units of $\quat$
and let $\G(\whZ)$  be the image in $\G(\Af)$ of the open compact subgroup $\what{\mcO}^{\times}\subset\quatt(\Af).$ There is a natural identification
\begin{equation}\label{eq:identification}
 \psi_p: \SS_p\simeq\lrquot{\G(\bQ)}{\G(\adelef)}{\G(\whZ)} \simeq \lrquot{\G(\bQ)}{\G(\adele)}{\G(\bR)\G(\whZ)}.
\end{equation}
\end{proposition}


\begin{proposition}[{cf.~\cite[Lem.~42.1.11]{voight}}]\label{propSSisog} The set $\SS_p$ forms a single isogeny class. 
\end{proposition}

\subsection{Supersingular curves and ideal classes}\label{sec:ideal class}
\begin{definition}
A (fractional) left-$\mathcal{O}$-ideal $I \subset \quat(\Q)$ is a lattice, invariant under multiplication on the left by $\mathcal{O}$. We denote by $\mcI(\mcO)$ the set of left-$\mathcal{O}$-ideals.

Two such ideals $I,J$ are $\quatt$-homothetic if and only if there is some $z\in\quat^\times(\bQ)$ such that $J=Iz$. We denote  by $$\Cl(\mathcal{O})=\mcI(\mcO)/\quat(\Qq)^\times$$ the set of $\quatt$-homothety classes of left-$\mathcal{O}$-ideals and by $h(\mcO)=|\Cl(\mcO)|$ its cardinality.
\end{definition}

Let $E_0$ be a supersingular curve over $\overline{\bF_p}$ and let $\mcO$ and $\quat$ as defined above.

\begin{proposition}\label{prop:correspEll} 
There is a one-to-one correspondence 
$$
\SS_p\simeq \mcI(\mcO)/\quat(\Qq)^\times=\Cl(\mcO).	
$$
\end{proposition}

\begin{proof}[Sketch of Proof]
We briefly recall the principle of the proof and refer to \cite[\S 3]{waterhouse} and \cite[Ch.~42]{voight} for details.

In one direction, the map is given as follows: given $E/\ov\Fp$ a supersingular elliptic curve and $\vphi:E\to E_0$ an isogeny the set $\Hom(E,E_0)$  is a left $\mcO$-module (via post-composition by isogenies of $E_0$) and embeds into $\mcO$ via
$$\psi\in\Hom(E,E_0)\hookrightarrow \psi\hat\vphi\in\mcO$$
where $\hat\vphi:E_0\mapsto E$ is the dual of $\vphi$.
Moreover, its image $I_{E,\vphi}=\Hom(E,E_0)\hat\vphi$ contains $\mcO\vphi\hat\vphi=\deg(\vphi)\mcO$ so $I_{E,\vphi}$ is a left-$\mathcal{O}$-ideal. 
One then checks that its homothety class is independent of the choice of $\vphi$.  

In the reverse direction, given a  left-$\mathcal{O}$-ideal $I \subset \quat(\Q)$, up to multiplying $I$ by some integer $n\not=0$, we may assume that $I\subset \mcO$. 
Let $$H(I)=\bigcap_{\phi\in I}\ker\phi\subset E_0.$$ 
Then $H(I)$ is a finite subgroup scheme and the quotient 
\begin{equation}\label{EIdef}
	E_I=E_0/H(I)
\end{equation}
 is an elliptic curve satisfying 
$\Hom(E_I,E_0)\simeq I.$
This isomorphism and the fact that these two maps are well-defined and inverses of one another (upon taking isomorphism and homothety classes) follow from \cite[Thm.~3.11]{waterhouse}; the key point is that, since $\mcO$ is a maximal order, the $\mcO$-ideals above are {\em kernel ideals} (\cite[Thm.~3.15]{waterhouse}).
\end{proof}


\begin{proof}[Proof of Proposition~\ref{prop:Bpinfty}]
Since $\mcO$ is maximal, any left $\mcO$-ideal $I$ is locally principal (see \cite[Thm.~16.1.3]{voight}): for any prime $\ell$
$$I_\ell:=I\otimes\Zl=\mcO_\ell \beta_\ell$$
for some $\beta_\ell\in\quat(\Qq_\ell)^\times$ uniquely defined modulo (left multiplications by) $\mcO_\ell^\times$ where $\mcO_\ell = \mcO \otimes_\Z \Z_\ell$.
By the local-global principle for lattices, the map
$$\beta_\mathrm{f}=(\beta_\ell)_\ell\in\quatt(\Af)\mapsto (\mcO_\ell\beta_\ell^{-1})_\ell=(I_\ell)_\ell\mapsto I$$ yields a  bijection
\begin{align}\label{eq:O-ideals adelically}
\mcI(\mcO)\simeq \quatt(\Af)/\what\mcO^\times.
\end{align}
Thus, we obtain
$$\Cl(\mcO)\simeq \quatt(\Qq)\bash \quatt(\Af)/\what\mcO^\times.
$$
It follows from $\Af^\times=\Q^\times\widehat{\Zz}^\times$ that
the projection $\quat^\times\to \G$ induces a bijection
\begin{equation}\label{eq:Cl(O) adelically}
\Cl(\mcO)
\simeq
\lrquot{\G(\bQ)}{\G(\adelef)}{\G(\whZ)}\simeq
\lrquot{\G(\bQ)}{\G(\Aa)}{\G(\Rr)\G(\whZ)}
\end{equation}
where the class of $I = \widehat{\mcO}\beta \cap \quat(\Q) \subset \mcO$ is mapped to the class of $\beta^{-1}$. 
Proposition \ref{prop:correspEll} concludes the proof.
\end{proof}

\subsection{A natural measure on $\Cl(\mcO)$} The identification \eqref{eq:identification} provides the finite set $\Cl(\mcO)$ with a natural probability measure $m_{\Cl(\mcO)}$ deduced from the Haar measure $m_{[\G]}$ on $[\G]$. In this section we compute the mass of the various ideal classes $[I]\in\Cl(\mcO)$.

To describe the measure we will need the following notion.

\begin{definition}
For any left-$\mathcal{O}$-ideal $I$ we define its right-order by
\begin{align*}
\mcO_R(I) = \{\beta \in \quat(\Q): I\beta \subseteq I\}.
\end{align*}
\end{definition}

If $I,J$ are two left-$\mathcal{O}$-ideals  in the same homothety class, their right-orders are clearly $\quat^\times(\Q)$-conjugate.
More generally, if 
$$I=\alpha_{\mathrm{f}}\star \mcO
:=\widehat{\mathcal{O}}\alpha_\mathrm{f}^{-1} \cap \quat(\Q)$$
as in \eqref{eq:O-ideals adelically} for some $\alpha_{\mathrm{f}}=(\alpha_q)_q\in \quat^\times(\adelef)$ then
\begin{equation}\label{rightorderconj}
  \mcO_R(I) = \alpha_{\mathrm{f}}\widehat{\mathcal{O}}{\alpha_{\mathrm{f}}}^{-1} \cap \quat(\Q).
\end{equation}
In particular, right-orders are locally maximal hence maximal. 

Notice also that the action by conjugation of
 $\quat^\times(\Af)$ as in \eqref{rightorderconj} is transitive on the set of maximal orders (for every prime $\ell$, $\GL_2(\Q_\ell)$ acts transitively on the maximal orders of $\Mat_2(\Q_\ell)$ and a division algebra over the local field $\Q_\ell$ has only one maximal order). This implies that any maximal order of $\quat(\Q)$ is the right-order of some (possibly non-unique) left-$\mathcal{O}$-ideal.


\begin{lemma}[Mass distribution]\label{lem:massdistribution}
For any $[I]\in\Cl(\mcO)$ we have
$$
    m_{\Cl(\mcO)}([I])
    =\frac{{|\units{\mcO_R(I)}/\{\pm 1\}|^{-1}}}{\sum_{[J] \in \Cl(\mathcal{O})}{|\units{\mcO_R(J)}/\{\pm 1\}|^{-1}}}=\frac{{|\units{\mcO_R(I)}|^{-1}}}{\sum_{[J] \in \Cl(\mathcal{O})}{|\units{\mcO_R(J)}|^{-1}}}.
 $$
\end{lemma}

As noted in the introduction, Eichler's mass formula \eqref{eq:Eichler mass formula} states that
\begin{align*}
\sum_{[J] \in \Cl(\mathcal{O})}\frac{1}{|\units{\mcO_R(J)}/\{\pm 1\}|} 
= \frac{p-1}{12}.
\end{align*}

\begin{proof}
Recall that $\quat$ is definite and therefore $\G(\Rr)$ is compact. 
Write $I=\alpha_{\mathrm{f}}\star\mcO$ for some $\alpha_{\mathrm{f}}\in\quat^\times(\Af)$. 
By definition, $m_{\Cl(\mcO)}([I])$ is the $m_{[\G]}$-measure of the coset 
$$[\alpha_{\mathrm{f}} \G(\Rr)\G(\whZ)]=\G(\Q)\bash\G(\Q)\alpha_{\mathrm{f}} \G(\Rr)\G(\whZ)\subset[\G]$$
or, by invariance of $m_{[\G]}$, equal to $m_{[\G]}([ \G(\Rr)\alpha_{\mathrm{f}}\G(\whZ){\alpha_{\mathrm{f}}}^{-1}])$.
Observe that
\begin{align*}
\widehat{\mcO_R(I)} = \alpha_{\mathrm{f}}\widehat{\mathcal{O}}{\alpha_{\mathrm{f}}}^{-1}
\quad \text{and}\quad
 \alpha_{\mathrm{f}}\G(\whZ){\alpha_{\mathrm{f}}}^{-1}
 = \punitsadelef{\mcO_R(I)}.
\end{align*}
Now notice that the stabilizer of the identity coset under the right-action of $\alpha_{\mathrm{f}}\G(\whZ){\alpha_{\mathrm{f}}}^{-1}$ is
\begin{align*}
\G(\Q) \cap \alpha_{\mathrm{f}}\G(\whZ){\alpha_{\mathrm{f}}}^{-1} = \punits{\mcO_R(I)} := \mcO_R(I)^\times/\{\pm1\}
\end{align*}
Therefore, using unimodularity of $\G(\adele)$ we have
\begin{align*}
m_{\Cl(\mcO)}([I]) = \frac{m_{\G(\adele)}(\G(\Rr)\alpha_{\mathrm{f}}\G(\whZ){\alpha_{\mathrm{f}}}^{-1})}{|\punits{\mcO_R(I)}|}= \frac{m_{\G(\adele)}(\G(\Rr)\G(\whZ))}{|\punits{\mcO_R(I)}|}.
\end{align*}
Summing over all $[I]\in \Cl(\mcO)$ (see \eqref{eq:Cl(O) adelically}) and using that $m_{[\G]}$ is a probability measure, we see that 
$$m_{\G(\adele)}(\G(\Rr)\G(\whZ))^{-1}=\sum_{[I] \in \Cl(\mathcal{O})}\frac{1}{|\punits{\mcO_R(I)}|}.$$
Thus, $m_{\Cl(\mcO)}([I])$ is the expression claimed.
\end{proof}

\subsubsection{Interpretation in terms of elliptic curves}\label{sec:massdistr for supersingular}
By Proposition~\ref{prop:correspEll},
we have an identification
$$\Cl(\mcO)\simeq \SS_p.$$
Let us look at what Lemma~\ref{lem:massdistribution} means in terms of elliptic curves. 
Let $$\{I_0=\mcO, \cdots, I_{h(\mcO)-1}\}$$ be representatives of the various ideal classes in $\Cl(\mcO)$ with $I_i\subset\mcO$. 
For every $i=0,\cdots,h(\mcO)-1$ let $E_i$
be the elliptic curve corresponding to $I_i$ constructed in the proof of Proposition~\ref{prop:correspEll}. 
By \cite[Prop.~3.9]{waterhouse} we have $\End(E_i)\simeq \mcO_R(I_i).$ 
Hence, for any $E\in\SS_p$
$$m_{\SS_p}(E)=\frac{|\Aut(E)|^{-1}}{\sum_{E'\in\SS_p}|\Aut(E')|^{-1}} = \nu_p(E)
$$
where $\nu_p$ is defined in the introduction.

\subsection{Supersingular reduction of CM elliptic curves}
\label{sec:supersingularreduction}

Let $\field$ be an imaginary quadratic field and $\order \subset \field$ an order.
As in the introduction, we fix for any $p$ an embedding $\Qq\hookrightarrow\overline{\Qp}$. 
This choice determines a place of $\ov\Qq$ above $p$ which we denote by $\mfp$. 
For any subextension $k\subset\ov\Qq$ we denote by $\order_{k,\mfp}$ the completion of $\order_k$ at $\mathfrak{p}$ and write $\ovorderp$ for $\order_{\ov\Qq,\mfp}$.

Given $E$ an elliptic curve with CM by $\order$ we may assume that $E$ is defined over the ring class field $\ringclassfield{\order}$ of $\order$ and in particular over $\ov\Qq$. 
By the work  of Serre-Tate \cite[p.~506-507]{SerreTate_goodreduction}, $E$ has potential good reduction everywhere. 
In particular, there is an elliptic curve $E'$ defined over $\ringclassfield{\ordernumberfield}$ with good reduction at $\primeideal$ and isomorphic to $E$ over $\ov\Qq$ (i.e.~an $\ringclassfield{\order}$-form of $E$): there exists a smooth model $\mcE'$ defined over $\order_{\ringclassfield{\order},\mfp}$ whose generic fiber is $E'$ and whose special fiber at $\mfp$ is a certain elliptic curve denoted by $\mcE'\mod\mfp$ and defined over the residue field of $\order_{\ringclassfield{\order},\mfp}$. 
Its $j$-invariant $j(E)$ is an algebraic integer \cite[Ch.~II, \S~6]{silverman-advanced} (this is a consequence of the everywhere potential good reduction of $E$) and we have 
$$j(\mcE'\mod \mfp)=j(E)\mod\mfp.$$
In particular, the isomorphism class of the reduction does not depend on the choice of the form $E'$. We will thus simply speak of {\em the reduction modulo $\mfp$} of the isomorphism class $E\in\CM_\order$ and will sometimes denote the class of $\mcE'\mod\mfp$ by
$E\mod \mfp.$
%

\begin{remark}\label{multiplegood} 
In fact, given $s$ distinct primes $p_1,\ldots,p_s$, $\mfp_1,\ldots,\mfp_s$ a choice of places above these primes as before, and an elliptic curve $E$ with CM by $\order$, there exists $E'$ isomorphic to $E$ over $\ov\Qq$ such that $E'$ has good reduction at every $\mfp_i$ (\cite[Cor.~1, p.~507]{SerreTate_goodreduction}).
\end{remark}

From now on we assume that $E$ has good reduction at $\mfp$ and we denote by $\mcE$ its N\'eron model (which we take over the larger ring $\ovorderp$ to avoid any rationality issues).


\begin{lemma}\label{lem:inert implies supersingular red.}
Let $E \in \CM_\ordernumberfield$. For $p$ inert in $\field$, the elliptic curve $E\mod \mfp$ is supersingular.
\end{lemma}

\begin{proof}
This is a result of Deuring \cite[p.~203]{deuring41} and a modern reference is \cite[Ch.~13,~Thm.~12]{lang-elliptic}.
\end{proof}

We assume now that $p$ is inert in $\field$.
Denote the maximal order and the quaternion algebra attached to $E\mod\mfp=\mcE\mod\mfp$ by
\begin{align}\label{eq:def O_E}
\mcO_E=\End(\mcE\mod \mfp)\subset 
\quat_E=\End(\mcE\mod \mfp)\otimes_\Zz\Qq\simeq\quat_{\infty,p}.
\end{align}
We therefore have an identification
\begin{equation}\label{correspEllpCl}
\SS_p\simeq \Cl(\mcO_E) = \mcI(\mcO_E)/ \quatt_E(\Qq)	
\end{equation}
under which the class of $\mcE\mod\mfp$ corresponds to the ideal class  $[\mcO_E]$.

\subsection{Compatibility with the Picard group action}\label{sec:comp Picard action} 
As in the previous section, let $E$ have CM by $\order$ and let $\mcE$ be its N\'eron model.
By general properties of the N\'eron model, the reduction modulo $\mfp$ induces an embedding
$$
	\iota_{E}:\End_{\ov\Qq}(E)\simeq\End_{\ovorderp}(\mcE)=\order \hookrightarrow \End_{\ov\Fp}(E\mod \mfp)=\mcO_E,
$$
and by extension (tensoring with $\Qq$) a $\Qq$-algebra embedding
\begin{equation}\label{iotaEdef}
\iota_{E}:\field\hookrightarrow \quat_E.
\end{equation}

\begin{proposition}\label{propoptimalss} 
Assume that $p$ does not divide the conductor of $\order$. 
The embedding $\iota_{E}$ of $\order$ is optimal, that is, one has
$$\iota_{E}(\field)\cap\mcO_E=\iota_{E}(\order).$$
\end{proposition}

\begin{proof}
See \cite[Prop 2.2]{LV}. It is proven that the embedding is locally optimal at every prime $\not=p$ and it is optimal at $p$ if and only if $\order_p$ is the maximal order of $\field_p$.
The latter is equivalent to our assumption.
\end{proof}

We have seen that all elliptic curves with CM by $\order$ are of the form (up to isomorphism)  
$\mfa\star E$ where $\mfa\subset\order$ is a proper $\order$-ideal. The next proposition states that this presentation is compatible with reduction modulo $\mfp$.

\begin{proposition}\label{propcompatibility} 
In the identification \eqref{correspEllpCl} the isomorphism class of $\mfa\star E\mod \mfp$ corresponds to the class of the left-$\mcO_E$ module $\mcO_E\,\iota_{E}(\mfa).$
\end{proposition}

The proof of this Proposition, which one can find\footnote{The proof is given for $\order=\order_\field$ the maximal order but it carries over to general orders of $\field$ since quadratic orders are Gorenstein rings.} in \cite[Thm.~7.12]{conrad}, is discussed in the section below. 
The main ingredient is a construction due to Serre \cite{SerreCM}: the {\em $\mfa$-transform}. We briefly recall its definition and refer to \cite[\S 7]{conrad}, \cite{giraud} and \cite[Chap.~7]{milne} for more details.

\subsection{The $\mfa$-transform} 

Let $A$ be a (not necessarily commutative) associative ring with unit, $M$ be a projective left $A$-module of finite type and constant rank $r$, $S$ a scheme and $E/S$ a (smooth) elliptic scheme over $S$ on which $A$ acts (on the left) by $S$-isogenies.  
The functor
$$M\star E: T\rightsquigarrow \Hom_A(M,E(T))$$
on the category of schemes over $S$  is representable by an abelian group scheme over $S$ of relative dimension $r$ and equipped with an $A$-action.

\begin{example}\label{atransformex1}
	For $A=\order$, $M=\mfa\in\mcI(\order)$ a proper ideal, $S=\spec(\ovorderp)$, $\mcE/S$ the N\'eron model of an elliptic curve $E\in\CM_\order$ with good reduction at $\mfp$, we have (\cite[Thm.~7.6]{conrad})
$$\mfa\star\mcE(\Cc)=\mfa\star E(\Cc)\simeq \Cc/\Lambda\mfa^{-1}.$$
\end{example}

\begin{example}
Given $E_0\in\SS_p$, 
$$A=\End_{\ov\Fp}(E_0)=\mcO_{E_0}\subset\End^0_{\ov\Fp}(E_0)=\quat_{E_0}$$ 
a maximal order and $I\subset\mcO_{E_0}$ a left $\mcO_{E_0}$-module and $S=\spec(\ov\Fp)$, we have (see \eqref{EIdef})
$$I\star {E_0}=E_I=E_0/H(I).$$
\end{example}

In the case of Example \ref{atransformex1}, after translating the notations of \cite{conrad} into ours, the formula in \cite[Thm.~7.12]{conrad} is the isomorphism of $\mcO_E$-modules
$$\Hom_{\ov{\Fp}}(\mfa\star\mcE\mod\mfp,\mcE\mod\mfp)\simeq\mcO_E\, \iota_{E}(\mfa).$$
This is the claim in Proposition~\ref{propcompatibility}.

\subsection{Compatibility with the Picard group action in adelic terms} In this section we interpret Proposition~\ref{propcompatibility} and therefore the map 
\begin{align*}
\red_p: \CM_\order \to \SS_p,\ E \mapsto E \mod \mfp
\end{align*}
(see \S\ref{sec:supersingularreduction}) in adelic terms.
Here, $\order \subset \field$ is a quadratic order, $p$ is inert in $\field$ (see Lemma~\ref{lem:inert implies supersingular red.}). 
We assume that $p$ does not divide the conductor of $\order$ as in Proposition~\ref{propoptimalss}.

Given $E\in\CM_\order$ we let $\mcO_E$ and $\quat_E$ be the maximal order and the quaternion algebra associated to $E$ as in \eqref{eq:def O_E} and let $\G_E = \PBt_E =\quatt_E/\Gm$ be the projective group. 
Thus, we have an identification
\begin{equation}\label{SSpidentwrtE}
\psi_{E,p}:	\SS_p\simeq \quatt_E(\Qq)\bash\quatt(\Af)/\whO_E^\times\simeq \G_E(\Qq)\bash \G_E(\Af)/\G_E(\whZ)
\end{equation} 
by Proposition \ref{prop:Bpinfty} where the isomorphism class of $E\mod \mfp$ corresponds to the class of the identity element and where $\G_E(\whZ) = \punitsadelef{\mcO_E}$.

The embedding $\iota_E$ in \eqref{iotaEdef} induces an embedding of $\Qq$-algebraic groups
$$\iota_E:\Res_{\field/\Qq}(\Gm)\hookrightarrow \quatt_E$$
and by optimality (Proposition \ref{propoptimalss}) we have $$\iota_E(\whorder ^\times)=\iota_E(\KAf)^\times\cap \what\mcO_E^\times.$$
Projecting to $\G_E$, we obtain a torus embedding 
$$\iota_E:\torus_\field=\Res_{\field/\Qq}(\Gm)/\Gm\to \G_E$$
whose image we denote by $\torus_E$.
By \S\ref{sec:CM as adelic quot} we have
\begin{align}\label{eq:Picasprojdoublequot}
\Pic(\order)&\simeq \fieldt\bash \fieldt(\adelef)/\whorder^\times
\simeq \torus_\field(\Qq)\bash \torus_\field(\Af)/\mathrm{P}\what\order^\times \\ 
&\simeq \torus_E(\Qq)\bash \torus_E(\Af)/\torus_E(\whZ)  \nonumber
\end{align}
where $\torus_E(\whZ):=\torus_E(\Af)\cap\G_E(\whZ)$.

Let $\mfa\in\mcI(\order)$ and $t_\mathrm{f}\in\fieldt(\Af)$ be related by \eqref{aidele} (i.e.~such that
$\what\mfa=t_\mathrm{f}\whorder$) and let $\ov t_\mathrm{f}$ be the image of $t_\mathrm{f}$ in $\torus_\field(\Af)$. 
By Proposition \ref{propcompatibility} the class of $\mfa\star E\mod\mfp$ corresponds to the left $\mcO_E$-module
$$\mcO_E\,\iota_E(\mfa)=\quat_E(\Qq)\cap\what\mcO_E\,\iota_E(t_\mathrm{f}) $$
and therefore to the classes  $[\iota_E(t_\mathrm{f})^{-1}]$ and $[\iota_E(\ov t_\mathrm{f})^{-1}]$ respectively under the identifications in \eqref{SSpidentwrtE}.
To summarize, we have shown that
\begin{align}\label{eq:basepointdep.psi_p}
\psi_{E,p}(\mfa \star E \mod \mfp) = [\iota_E(\overline{t}_\mathrm{f}^{-1})].
\end{align}

\subsubsection{Changing the reference curve}\label{sec:changing ref curve}
A minor point is that the identification \eqref{SSpidentwrtE} is made with respect to the reference curve $E\mod \mfp$ which is a priori varying with $\order$. 
Since $p$ is fixed and the space $\SS_p$ is finite we could resolve the issue by passing to subsequences of orders $\order$ such that all the curves $E_\order$ have a given reduction modulo $\mfp$. However, it is  more natural to keep track of  \eqref{SSpidentwrtE} when we use \eqref{eq:identification} with a fixed reference curve $E_0\in\SS_p$. 

Let $E_0\in\SS_p$ with associated quaternion algebra, maximal order and projective group denoted by $\quat,\ \mcO$ and $\G$ respectively. 

Let $E\in\CM_\order$ and 
$$\vphi:E \mod \mfp\to E_0$$ an isogeny between the two curves. We have an embedding
$$f_{\vphi}:\psi\in\End(E\mod\mfp)=\mcO_E\hookrightarrow 
\frac{1}{\deg(\vphi)}\vphi\psi\hat\vphi\in\quat(\Qq),$$
which extends to an isomorphism of $\Qq$-algebras
$$f_{\vphi}:\quat_E\simeq \quat.$$
This also gives a bijection
$$f_{\vphi}:\G_E(\Qq)\bash\G_E(\Af)/\G_E(\whZ)\simeq \G(\Qq)\bash\G(\Af)/f_{\vphi}(\G_E(\whZ)).$$
The right order of the left $\mcO$-ideal
$$I_{E,\vphi}:=\Hom(E\mod \mfp,E_0)\hat\vphi\subset\mcO = \End(E_0)$$
is the maximal order $f_{\vphi}(\mcO_E)$.
One can check (see the remark below) that its (right) $\quatt(\Qq)$-homothety class does not depend on $\vphi$ and represents (the isomorphism class of) $E\mod \mfp$ under Proposition~\ref{prop:correspEll}.

Let $\beta_{E,\mathrm{f}}\in \quatt(\Af)$ such that 
\begin{align}\label{eq:changing ref curve}
\what I_{E,\vphi}=\what\mcO\beta_{E,\mathrm{f}}^{-1}.
\end{align}
Since the right order of $I_{E,\vphi}$ is $f_{\vphi}(\mcO_E)$ have
$$f_{\vphi}(\what\mcO_E)=\beta_{E,\mathrm{f}}\what\mcO\beta_{E,\mathrm{f}}^{-1}$$
and hence $$f_{\vphi}(\G_E(\whZ))=\beta_{E,\mathrm{f}}\G(\whZ)\beta_{E,\mathrm{f}}^{-1}.$$
From this we deduce that the map
$$\G(\Qq)\alpha_\mathrm{f}f_{\vphi}(\G_E(\whZ))\mapsto 
\G(\Qq)\alpha_\mathrm{f}\beta_{E,\mathrm{f}}\G(\whZ)$$
is a bijection
\begin{equation}\label{bijEE0}
\G(\Qq)\bash\G(\Af)/f_{\vphi}(\G_E(\whZ))\simeq 
\G(\Qq)\bash\G(\Af)/\G(\whZ).	
\end{equation}

\begin{remark}
If we choose a different isogeny $\vphi'$, the Skolem-Noether theorem implies that 
%
the images of $\mcO_E$ by  $f_{\vphi}$ and $f_{\vphi'}$ are conjugate and the modules $I_E$ and $I'_E$ vary accordingly.
\end{remark}
\awnote{skolem noether}

Let $\iota:\field \to \quat$ be the composition of the embedding $\iota_E: \field \to \quat_E$ with $f_\varphi$.
This defines an embedding
\begin{equation}\label{iotapdef}
\iota:\torus_\field\hookrightarrow \G. 
\end{equation}
The optimality property of $\iota_E$ translates into
\begin{align}\label{eq:optimal at p}
\iota(\whorder ^\times)=\iota(\KAf)^\times\cap \beta_{E,\mathrm{f}}\what\mcO^\times\beta_{E,\mathrm{f}}^{-1}.
\end{align}

Combining \eqref{eq:basepointdep.psi_p} with the bijection \eqref{bijEE0} we obtain the analogue of Proposition~\ref{propredinfty} at the place $p$.
 
\begin{proposition}\label{propredp} 
Let $E=E_\order\in\CM_\order$ and $\iota$ the embedding constructed above. In the identification \eqref{eq:identification} the set $\red_p(\CM_\order)$ is represented by the projection of the adelic  torus orbit $\G(\Qq)\iota(\torus_\field)(\Aa)\beta_{E,\mathrm{f}}$. That is,
$$\psi_p(\red_p(\CM_\order))=[\iota(\torus_\field)(\Aa)\beta_{E,\mathrm{f}}]
=\G(\Qq)\iota(\torus_\field)(\Aa)\beta_{E,\mathrm{f}}\G(\Rr)\G(\whZ).$$
More precisely, if $\mfa\in\mcI(\order)$ and $t_\mathrm{f}\in\fieldt(\Af)$ are related by \eqref{aidele}, then $\red_p(\mfa\star E)=\mfa\star E\mod\mfp$ is represented by the class 
$$\psi_p(\red_p(\mfa\star E))=[\iota(\ov t_\mathrm{f}^{-1})\beta_{E,\mathrm{f}}]
=\G(\Qq)\iota(\ov t_\mathrm{f}^{-1})\beta_{E,\mathrm{f}}\G(\Rr)\G(\whZ).$$
\end{proposition}

\section{Diagonal compatibility}\label{sec:diagonalComp}
In this section we combine Propositions \ref{propredinfty} and \ref{propredp} as already outlined in \S\ref{sec:intro-reductionstep}.
This continues the reduction of Theorem~\ref{thm:main-elliptic} to Theorem~\ref{thm:main-dynamic}.

We fix $p_{1},\ldots,p_{s}$ distinct odd primes and for each $i=1,\ldots,s$ we fix a supersingular elliptic curve $E_{0,i}\in\SS_{p_i}$ in characteristic $p_i$. 
Each $E_{0,i}$ determines a quaternion algebra $\quat_i\simeq \quat_{p_i,\infty}$ and a maximal order $\mcO_i$ in it. 
We denote by $\G_{i}=\PBt_i$ the corresponding projective group of units. Furthermore, we let $\quat_0=\Mat_2$ be the split quaternion algebra, $\mcO_0=\Mat_2(\Zz)$ and $\G_0=\PGL_2$. 
We also define for $i=0,\ldots, s$ the compact subgroups
$$K_i=K_{\infty,i}K_{\mathrm{f},i}\subset\G_i(\Aa)$$
where $K_{\infty,i}=\G_i(\Rr)$ for $i \geq 1$, $K_{\infty,0}=\mathrm{P}\SO_2(\Rr)$ and where
$$K_{\mathrm{f},i}=\G_i(\whZ):=\mathrm{P}\what\mcO_i^\times=\what\mcO_i^\times/\what\Zz^\times.$$
for all $i$. For the product group
$$\G:=\prod_{i=0}^s\G_i$$
we define the compact subgroups
\begin{align*}
&K=K_\infty\Kf = \prod_{i=0}^sK_i \subset\G(\Aa),\\ 
&K_\infty=\prod_iK_{\infty,i},\text{ and } \Kf=\prod_i K_{\mathrm{f},i}.
\end{align*}

The identifications \eqref{identcomplexadelic} for $i=0$ and \eqref{eq:identification} for $i=1,\ldots,s$ combine into an identification
\begin{equation}\label{productidentification}
\psi:\Ell_\infty\times\prod_{i=1}^s\SS_{p_i}\simeq \G(\Qq)\bash \GA/K.
\end{equation}

Let $\field$ be an imaginary quadratic field in which for each $i=1,\ldots,s$ the prime $p_i$ is inert. Let $\order\subset\field$ be an order and assume that the primes $p_1,\ldots,p_s$ do not divide the conductor of $\order$.
The various reduction maps $\red_{p_i}$ for $i=0,\ldots,s$ combine into a multi-reduction map
$$\red_{\bfp}:E\in\CM_\order\mapsto (\red_{p_i} E)_{i=0,\ldots,s}\in \Ell_\infty\times\prod_{i=1}^s\SS_{p_i}.$$
The choice of the CM curve $E_\order\simeq \Cc/\order$ provides a parametrization of $\CM_\order$ by $\Pic(\order)$ via 
$$[\mfa]\mapsto\mfa\star E_\order=E_{\mfa^{-1}}$$
and a map
\begin{equation}\label{mfamap}
[\mfa]\in \Pic(\order)\mapsto \red_\bfp(\mfa\star E_\order)=\red_\bfp(E_{\mfa^{-1}}).	
\end{equation}
In addition, the choice of $E = E_\order$ determines embeddings $\iota_i:\torusK\hookrightarrow \G_i$ (see \eqref{iotapdef}) which combine into a diagonal embedding
\begin{equation}\label{iotatotal}
\iota=(\iota_i)_{i=0,\ldots,s}:\torusK\hookrightarrow \G=\prod_{i=0}^s\G_i
\end{equation}
whose image we denote by
$$\torus_{\iota}:=\iota(\torus_{\field}).$$

We let $\mathrm{g} = \mathrm{g}_\infty \mathrm{g}_\mathrm{f} \in \G(\Aa)$ be given by 
\begin{align}
\rmg_\infty &= (\ov g_\infty,\Id,\cdots,\Id)\in\G(\Rr)\label{eq:rmg real}\\
\rmg_\mathrm{f} &= (\Id, \beta_{\mathrm{f},1},\ldots, \beta_{\mathrm{f},s})\in \G(\Af)\label{eq:rmg fiadelic}
\end{align}
where $g_\infty\in\GL_2(\Rr)$ is defined above \eqref{conjiota} and where $\beta_{\mathrm{f},i}$ for $i\geq 1$ is defined in \eqref{eq:changing ref curve} given the curve $E= E_\order$.
By definition of $\ov g_\infty$ and of $K_\infty$ we have
$$
\torus_{\iota}(\Rr)<\rmg_\infty K_\infty \rmg_\infty^{-1}.	
$$
By \eqref{optimapeq} and \eqref{eq:optimal at p}, the embedding $\iota$ is {\em optimal} in the sense that
$$
\iota(\mathrm{P}\whOrt)=\torus_{\iota}(\Af)\cap \rmg_\mathrm{f} \Kf \rmg_\mathrm{f}^{-1}.	
$$

Combining Proposition \ref{propredinfty} and Proposition \ref{propredp} we obtain that the map \eqref{mfamap} admits the following adelic description:

\begin{proposition}\label{propred mfp}
In the identification \eqref{productidentification}, the set $\red_\bfp(\CM_\order)$ is represented by the projection of the adelic torus orbit $\G(\Qq)\torus_{\iota}(\Aa)\rmg$. That is,
$$\psi(\red_\bfp(\CM_\order))=[\torus_{\iota}(\Aa)\rmg]=\G(\Qq)\torus_{\iota}(\Aa)\rmg K.$$
More precisely, if $\mfa\in\mcI(\order)$ and $t_\mathrm{f}\in\fieldt(\Af)$ are related by \eqref{aidele}, then $\red_\bfp(\mfa\star E)$ is represented by the class 
$$\psi(\red_\bfp(\mfa\star E))=[\iota(\ov t_\mathrm{f}^{-1})\rmg]=\G(\Qq)\iota(\ov t_\mathrm{f}^{-1})\rmg K.$$
\end{proposition}

\section{Joint equidistribution for CM elliptic curves}\label{sec:Joint equi}
In this section we prove Theorem \ref{thm:main-elliptic} from the introduction assuming Theorem~\ref{thm:main-dynamic}. 
In fact, we prove the following generalisation.

\begin{theorem}\label{thm:main-ellipticorder}
Let $p_1,\ldots,p_s, q_{1},q_{2}$ be distinct odd primes. 
As $D\ra-\infty$  along the set of discriminants of imaginary quadratic orders $\order$ such that
  \begin{enumerate}
  \item every $p_i$ for $i\leq s,$ is inert in $\field=\Q(\sqrt D)$,
    \item $D$ is coprime to $p_i$ for all $i \leq s$, and
    \item $q_{1}$ and $q_{2}$ are split in $\field$ 
  \end{enumerate}
 the push-forward of the counting probability measure on $\CM_{\order}$ by $\mathrm{red}_\bfp$ equidistributes towards the product measure $\nu_\infty\otimes\nu_{p_{1}}\otimes\cdots\otimes\nu_{p_{s}}$.
 \end{theorem}
 
Given Proposition~\ref{propred mfp}, the proof is essentially an exercise in translating between the classical and the adelic language.

\subsection{Some notations}
\label{sec:proofofarithmeticapplication}
Let $p_1,\cdots,p_s, q_{1},q_{2}$ as above. We consider a sequence of imaginary quadratic orders
$$\order_n\subset\field_n,\ n\geq 1$$
of discriminant $D_n\ra-\infty$
such that $D_n$ satisfies (1)-(3) in Theorem \ref{thm:main-ellipticorder} for every $n$.
Let $$E_n:=E_{\order_n}\simeq \Cc/\order_n.$$ 
In the notation of the previous section, each curve $E_n$ gives rise to a diagonal embedding of $\torus_{\field_n}$
$$\iota_{n}=(\iota_{0,n},\ldots,\iota_{s,n}):\torus_{\field_n}\hookrightarrow\G$$ 
whose image we denote by $\torus_{\iota_n}$. Define $\rmg_n = \rmg_{\infty,n} \rmg_{\mathrm{f},n}$ as in \eqref{eq:rmg real} and \eqref{eq:rmg fiadelic} so that in particular
\begin{equation}\label{globaloptimalrealn}
\torus_{\iota_n}(\Rr)<\rmg_{\infty,n} K_\infty \rmg_{\infty,n}^{-1}.	
\end{equation}
We have the following optimality property
\begin{equation}\label{globaloptimaln}
T_{\iota_n,\rmf}:=\iota(\mathrm{P}\what\order^\times_n)=\torus_{\iota_n}(\Af)\cap \rmg_{\mathrm{f},n}\Kf	\rmg_{\mathrm{f},n}^{-1}.
\end{equation}
Finally, we define the compact subgroup
$$T_{\iota_n}:= \torus_{\iota_n}(\R)\times T_{\iota_n,\rmf}\subset \torus_{\iota_n}(\Aa).$$

\begin{remark}
Our assumption (3) on $q_{1}$ and $q_{2}$ implies that $\torus_{\iota_{n}}$ is split at $q_{1}$ and $q_{2}$.
\end{remark}

\subsection{Proof of Theorem \ref{thm:main-elliptic}}

It is sufficient to show that for any function $f_0\in\mcC_c(\Ell_\infty)$ and any tuple of super-singular elliptic curves $(\ov E_i)_{1\leq i\leq s}\in \prod_{i}\SS_{p_i}$ one has
\begin{multline}\label{classicallimit}
\frac{1}{|\Pic(\order_n)|}\sum_{[\mfa]\in\Pic(\order_n)}f_0(\Cc/\mfa^{-1})\prod_{i\geq 1}\delta_{\ov E_i}(\mfa\star E_n\mod\mfp_i)\\ \ra \nu_\infty(f_0)\prod_{i\geq 1}\nu_{p_i}(\ov E_i).	
\end{multline}
as $n\ra\infty$ where $\delta_{\ov E_i}$ denotes the Dirac function at the point $\ov E_i\in\SS_{p_i}$.
To prove this convergence, we will interpret the left- and the right-hand sides of \eqref{classicallimit} as adelic integrals.

\subsubsection{The right-hand side  of \eqref{classicallimit}} \label{sec:RHS classlim}
Under the identification \eqref{identcomplexadelic} the function $f_0$ correspond to a continuous compactly supported function on $[\G_0]_{K_0}$ or equivalently to a continuous compactly supported function $\tilde f_0$ (say) on $[\G_0]$ which is $K_{0}$-invariant.
In particular,
$$\int_{[\G_0]}\tilde f_0(g_0)dg_0=\nu_\infty(f_0).$$
Similarly, for $i\geq 1$ the Dirac function $\delta_{\ov E_i}$ corresponds to a $K_i=\G_i(\Rr)K_{\rmf,i}$-invariant continuous function on $[\G_i]$ under \eqref{eq:identification}, namely the characteristic function
of the class $$[g_{\rmf,\ov E_i}\G_i(\Rr)K_{\rmf,i}]=[g_{\rmf,\ov E_i}K_{i}]\subset [\G_i]$$
where the class of $g_{\rmf,\ov E_i}$ corresponds to $\ov E_i$ under \eqref{eq:identification}.
By Lemma~\ref{lem:massdistribution} and the discussion in \S\ref{sec:massdistr for supersingular}, we have
$$\int_{[\G_i]}1_{[g_{\rmf,\ov E_i}K_{i}]}(g_i)\, dg_i
=\nu_{p_i}(\ov E_i).$$
The product function
$$f=f_0\times\prod_{i=1}^s\delta_{\ov E_i}:\Ell_\infty\times \prod_{i\geq 1}\SS_{p_i}\ra \Cc$$
corresponds to the $K$-invariant function
$$\tilde f:=\tilde f_0\times\prod_{i\geq 1}1_{[g_{\rmf,\ov E_i}K_{i}]}:[\G]\ra\Cc$$
which satisfies
\begin{equation}\label{averagevalue}
	\int_{[\G]}\tilde f(\rmg)d\rmg= \nu_\infty(f_0)\prod_{i\geq 1}\nu_{p_i}(\ov E_i).
\end{equation}

\subsubsection{The left-hand side  of \eqref{classicallimit}}\label{sec:LHS classlim}
To ease notation, we will not display the index $n$ here.

Because of the obvious bijection
$$\torus_\iota(\Qq)\bash \torus_\iota(\Aa)\simeq \G(\Qq)\bash \G(\Qq)\torus_\iota(\Aa)=[\torus_\iota],$$
we view (by restriction) the shifted function on $[\G]$
$$\rmg.\tilde f(\cdot):t\mapsto \tilde f(t\rmg)$$ 
as a (continuous) function on $[\torus_\iota]$. 
By \eqref{globaloptimaln} and \eqref{globaloptimalrealn} this function is $T_{\iota}$-invariant, that is, it is constant along any set of the shape
$$[t_\mathrm{f}T_{\iota}]=\torus_\iota(\Qq)\bash \torus_\iota(\Q) t_\mathrm{f} T_{\iota}\text{ for } t_\mathrm{f}\in\torus_\iota(\Af).$$ 
Moreover, by invariance of the Haar measure we have
\begin{align*}
\vol_{[\torus_\iota]}([t_\mathrm{f}T_{\iota}])
= \vol_{[\torus_\iota]}([T_{\iota}t_\mathrm{f}])
= \vol_{[\torus_\iota]}([T_{\iota}])
\end{align*}
as $\torus_\iota$ is abelian. Thus, all $T_\iota$-cosets have the same volume.
Hence,
\begin{align*}
\vol_{[\torus_\iota]}([t_\mathrm{f}T_{\iota}])
= \frac{1}{|[\torus_\iota]_{T_{\iota}}|}=\frac{1}{|\Pic(\order)|}
\end{align*}
where in the second equality we used \eqref{eq:Picasprojdoublequot} and \eqref{globaloptimalrealn}.
It follows that
\begin{align*}
\int_{[\torus_{\iota}]}\tilde f(t\rmg)dt&=\frac{1}{|[\torus_\iota]_{T_{\iota}}|}\sum_{[t_\mathrm{f}]\in [\torus_\iota]_{T_{\iota}}}\tilde f(t_\mathrm{f}\rmg)\\ &
=
\frac{1}{|\Pic(\order)|}\sum_{[\mfa]\in\Pic(\order)}f(\red_\bfp(\mfa^{-1}\star E))	\end{align*}
where the last equality follows from Proposition \ref{propred mfp}. 
This is equal to the left-hand side of \eqref{classicallimit}.

Therefore, we have for any $n$
$$\frac{1}{|\Pic(\order_n)|}\sum_{[\mfa]\in\Pic(\order_n)}f_0(\Cc/\mfa^{-1})\prod_{i\geq 1}\delta_{\ov E_i}(\mfa\star E_n\mod\mfp_i)=\int_{[\torus_{\iota_n}]}\tilde f(t\rmg_n)dt$$
which realizes the left-hand side of \eqref{classicallimit} as an adelic integral.

\subsubsection{The convergence claim in \eqref{classicallimit}}
By \S\ref{sec:RHS classlim} and \S\ref{sec:LHS classlim} it remains to show that as $n \to \infty$
\begin{align*}
\int_{[\torus_{\iota_n}]}\tilde f(t\rmg_n)\, dt \to \int_{[\G]}\tilde f(\rmg)\, d\rmg.
\end{align*}
Therefore, Theorem \ref{thm:main-elliptic} will follow from the equidistribution statement in Theorem \ref{thm:main-dynamic} once we have verified all its assumptions.

The condition \eqref{eq:assumption compactopen} is satisfied since $\mcO_i$ is a maximal order in $\quat_i$ for any $i$ (see the more general Lemma \ref{normsurjeichler}). 
It remains to see that the discriminant fulfills
\begin{align*}
\disc([\torus_{\iota_{n}}\rmg_{n}])
=\min_i \disc([\torus_{\iota_{i,n}}g_{i,n}])
\ra\infty
\end{align*}
as $n \to \infty$.
Here, $g_{i,n}$ is the projection of $\rmg_n$ to the $i$-th factor.
The {\em discriminant} of the torus orbit $\disc([\torus_{\iota_{i,n}}g_{i,n}])$ is defined in \cite[\S4]{dukeforcubic}, and one can show that there exists a constant $C>1$ independent of $n$ with
\begin{align*}
C^{-1} |D_n| \leq \disc([\torus_{\iota_{i,n}}g_{i,n}]) \leq C |D_n|.
\end{align*}
using \cite[Prop.~4.1]{ELMV-DukeJ} and \cite[Thm.~5.2(1)]{dukeforcubic} and the fact that $D_n$ is coprime to $p_i$.
Therefore
$$\disc([\torus_{\iota_{n}}\rmg_{n}])\ra\infty.$$
Thus, Theorem \ref{thm:main-dynamic} applies and concludes the proof of Theorem~\ref{thm:main-elliptic}.
\qed


\begin{remark}
Theorem \ref{thm:main-elliptic} can be refined further to provide more precise information by reducing the size of the compact subgroup $$K=\PSO_2(\Rr)\PGL_2(\whZ)\times\prod_{i=1}^s\G_i(\Rr)\G_i(\whZ).$$
For instance, we could replace for each $i=1,\ldots,s$ the full group $K_{\infty,i} = \G_i(\Rr)$ by $K_{\infty,i}=\G_i(\Rr)_{z_i}$, the stabilizer of a line 
$$\Rr z_{i}\subset \quat^0_i(\Rr)$$
in the quadratic space of real trace-zero quaternions on which $\G_i$ act by conjugation.

In that case $\G_i(\Rr)/K_{\infty,i}$ identifies with the ellipsoid of traceless quaternions of norm~$1$
$$\quat_i^{0,1}(\Rr)=\{z\in \quat_i(\Rr):\ \Tr_{\quat_i}(z)=0,\ \Nr_{\quat_i}(z)=1\}.$$
Theorem \ref{thm:main-dynamic} restricted to the $\prod_{i=0}^sK_{\infty,i}\Kf$-invariant functions and applied to sequences of torus orbits $[\torus_{n}\rmg_n]$ for suitable $\rmg_n$ may be interpreted as a joint equidistribution statement as $n\ra\infty$ for  primitive representations of $|D_n|$ by the genus classes of the ternary quadratic spaces $(\quat^0_i,\Nr_{\quat_i})_{i=0,\ldots,s}$. We leave it to the interested reader to work out this interpretation.
\end{remark}

\section{Invariance under the simply connected cover}\label{sec:simplyconnectedinvariance}

In this section we discuss the main step towards the proof of Theorem~\ref{thm:main-dynamic}, namely Theorem \ref{thm:invariance in product}; we briefly review the standard approach.
Let $[\torus_{\iota_n}\rmg_n]$ be a sequence of toral packets and denote by $\mu_n$ the Haar probability measure on $[\torus_{\iota_n}\rmg_n]$.
We analyse possible limits of the measures $\mu_n$.
To this end recall that the space of (Borel) measures $\nu$ on $[\G]$ with $\nu([\G]) \leq 1$ is compact (and metrizable) in the weak${}^\ast$-topology by a theorem of Banach and Alaoglu.
Let $(\mu_{n_k})_k$ be an arbitrary convergent subsequence of the sequence $(\mu_n)_n$. This means that there is a measure $\mu$ on $[\G]$ with $\mu([\G]) \leq 1$ and with
\begin{align*}
\int_{[\G]}f \,\mathrm{d} \mu_{n_k} \to \int_{[\G]} f \,\mathrm{d} \mu
\end{align*}
for all $f \in \cC_c([\G])$. Such a measure $\mu$ is called a weak${}^\ast$-limit of the sequence $(\mu_n)_n$.
To prove Theorem~\ref{thm:main-dynamic} we need to show that any such weak${}^\ast$-limit $\mu$ projects to the Haar measure $m_{[\G]_{\compactopenf}}$ on $[\G]_{\compactopenf}$, or equivalently, that for any $\compactopenf$-invariant function $f \in \cC_c([\G])$ we have   $\int_{[\G]} f \,\mathrm{d} \mu=\int_{[\G]} f \,\mathrm{d} m_{[\G]}$.

Theorem~\ref{thm:invariance in product} establishes that any $\mu$ as above, that is, any weak${}^\ast$-limit of the measures $m_{[\torus_{\iota_n}\rmg_n]}$ 
is a probability measure\footnote{
In other words, there is `no escape of mass'. If $[\G]$ were compact, this would be automatic. } and is invariant under the group of adelic points of the simply connected cover of $\G$.

In Section \ref{sec:ProofMainDyn} we explain how the surjectivity assumption \eqref{eq:assumption compactopen} allows us to deduce Theorem \ref{thm:main-dynamic} from Theorem \ref{thm:invariance in product}.
A key ingredient in the proof of the latter is 
a generalisation of Duke's result \cite{duke88}, cf.~Theorem~\ref{thm:dukethmforquaternionalgebras}, which roughly speaking  adresses Theorem~\ref{thm:invariance in product} in the case when $\G$ has one simple factor.
We then apply the joinings classification theorem of Einsiedler and Lindenstrauss \cite{EL-joining2} to deduce Theorem~\ref{thm:invariance in product}.

\subsection{Invariance for one factor}\label{sec:individual}
Let $\quat$ be a quaternion algebra defined over $\bQ$ (possibly the split algebra $\Mat_2$). Let $\G=\PBt$ be the associated projective group and let $\G(\Aa)^+<\G(\Aa)$ be the image of $\quat^1(\Aa)$ under the canonical projection of the subgroup
$\quat^{1}<\quatt$ of quaternions of norm $1$.
In the context of ergodic theory, $\quat^1$ has a special role as it is generated by unipotents over the algebraic closure.
The following theorem was stated (without proof\footnote{A proof will be given in the forthcoming book \cite{EPSbook}.}) in \cite[Thm.~4.6]{dukeforcubic}:

\begin{theorem}[Invariance in single factors]\label{thm:dukethmforquaternionalgebras}
  Let $\quat$ be a quaternion algebra over $\bQ$  and let $\G = \PBt$.
Let $[\torus_{\iota_n}g_n]$ be a sequence of toral packets where $\iota_{n}:\field_{n}\to\quat(\bQ)$ 
are embeddings of quadratic fields and where $g_{n}\in\G(\adele)$.
Assume that
  \begin{equation*}
    \mathrm{disc}\big([\torus_{\iota_{n}}g_{n}]\big)\to\infty\quad(n\to\infty).
  \end{equation*}
  Then any weak${}^\ast$-limit of the normalized Haar measures on the packets $[\torus_{\iota_{n}}g_{n}]$ is a probability measure invariant under $\G(\adele)^{+}$.
\end{theorem}

\begin{remark}
Theorem \ref{thm:dukethmforquaternionalgebras} is a broad generalisation of Duke's equidistribution theorems \cite{duke88} and a consequence of deep results by many authors \cite{DFI,Waldspurger,CU,MVIHES}. 
It builds on the theory of automorphic forms and associated $L$-functions, in particular on Waldspurger's formula and on subconvexity bounds. We invite the reader to look at the discussion surrounding \cite[Thm.~4.6]{dukeforcubic} for more details.
\end{remark}

\begin{remark}
For the purpose of proving Theorem \ref{thm:invariance in product}, one can also assume that a fixed prime $q_1$ splits in all the quadratic fields $\field_n$. 
Under this splitting assumption, one can then use (for certain natural sequences of toral packets) the modern presentations of the {\em ergodic method} of Linnik \cite{linnik} (for instance \cite{Linnikthmexpander}) to prove Theorem~\ref{thm:dukethmforquaternionalgebras}. This has been carried out by the last named author in \cite{W-Linnik}.
\end{remark}

\subsection{Joint invariance}\label{sec:joint}
We now upgrade Theorem \ref{thm:dukethmforquaternionalgebras} to multiple factors. 
Let us briefly recall the notation used in Theorem \ref{thm:main-dynamic}.
Let $\quat_{i}$ for $i=0,\ldots,s$ be a finite set of distinct quaternion algebras and let $\G_{i}$ be their of projective groups of units. 
Set $$\G = \G_0 \times \ldots \times \G_s$$
and denote by $\G(\Aa)^+ \subset \G(\Aa)$ the product of the groups $\G_i(\Aa)^+$ for $i=0,\ldots,s$.
Without loss of generality we may and will assume that $\G_0=\PGL_2$. Given a tuple $\iota = (\iota_0,\ldots,\iota_s)$ of embeddings $\iota_{i}:\field\to\quat_{i}$ of a given quadratic field $\field$, we obtain a tuple of morphisms of $\Q$-groups
\begin{align*}
\iota_i:\torus_\field = \Res_{\field/\Q}(\Gm)/\Gm \to \G_i
\end{align*}
We denote by
\begin{equation*}
\torus_\iota = \big\{(\iota_0(t),\ldots,\iota_s(t)): t \in \torus_\field\big\}
\end{equation*}
the diagonally embedded image.

\begin{theorem}[Invariance in the product]\label{thm:invariance in product}
Let $p_1,\ldots,p_s$ and $q_1,q_2$ be distinct odd primes. 
Let $\field_n$ for $n\geq 1$ be a sequence of quadratic fields
such that for every $n$
 \begin{enumerate}
 \item $p_1,\ldots,p_s$ are inert in $\field_n$, and
 	\item $q_1$ and $q_2$ are split in $\field_n$.
 \end{enumerate}
Let $\iota_n=(\iota_{i,n})_{i=0,\ldots,s}$ be a tuple of embeddings $\iota_{i,n}:\field_n\to\quat_i$ and let $\rmg_n \in \G(\adele)$.
If $$\disc([\torus_{\iota_n} \rmg_n]) \to \infty$$ as $n \to \infty$,
then any weak${}^\ast$-limit of the sequence of normalized Haar measures on the packets $[\torus_{\iota_n}\rmg_n]$ is a probability measure invariant under $\G(\adele)^{+}$.
\end{theorem}

\begin{proof}[Proof of Theorem \ref{thm:invariance in product}]
As the groups $\G_i$ are not isogenous over $\Q$, Theorem~\ref{thm:dukethmforquaternionalgebras} and \cite[Thm.~1.8]{EL-joining2} imply the theorem.
%
%
%
%
\end{proof}

\section{Proof of Theorem \ref{thm:main-dynamic}}\label{sec:ProofMainDyn}

In this section we upgrade Theorem \ref{thm:invariance in product} using the norm surjectivity assumption in \eqref{eq:assumption compactopen} to obtain Theorem \ref{thm:main-dynamic}.
We will also comment on stronger versions of Theorem \ref{thm:main-dynamic} in \S \ref{sec:refinements}.

\subsection{Conditioning on $\G(\adele)^+$-orbits}
We denote by $m$ the Haar probability measure on $[\G]$.
Let $\cA$ be the preimage of the Borel $\sigma$-algebra on $[\G]/\G(\adele)^+$.
In particular, $\cA$ is countably generated and its atoms are the $\G(\adele)^+$-orbits in $[\G]$.
As $\G(\adele)^+$ is normal in $\G(\Aa)$, $\cA$ is invariant under $\G(\adele)$.
Also, each $\G(\adele)^+$-orbit is closed since it is of the form
$$\G(\Q)\rmg\G(\adele)^+=\G(\Q)\G(\adele)^+\rmg$$ and $\G(\Q)\G(\adele)^+$ is closed (this follows from a theorem of Borel and Harish-Chandra \cite{borelharishchandra} applied to $\prod_i \quat_i^1$).
%

For $m$-almost every $x \in [\G]$ we may let $m_x^\cA$ be the conditional measure of $m$ on the atom of $x$.
As the atoms of $\cA$ are closed subsets, these measures are supported on the $\G(\adele)^+$-orbits (see \cite[Ch.~5]{vol1} for the definition and basic properties of conditional measures).
By $\G(\adele)^+$-invariance of $m$, it follows that $m_x^\cA$ is the $\G(\adele)^+$-invariant probability measure on the closed orbit $x\G(\adele)^+$ $m$-almost everywhere.
Using that, we may as well define $m_x^\cA$ to be the unique $\G(\adele)^+$-invariant probability measure on $x\G(\adele)^+$ \emph{everywhere}, preserving the defining property of conditional measures.

\subsubsection{The space of bounded uniformly continuous functions}

For what follows, it will be convenient to work with a certain class of test functions. A function $f:[\G]\to \C$ is \emph{uniformly continuous} if for all $\varepsilon>0$ there is a neighborhood $U\subseteq\G(\adele)$ of the identity such that
\begin{equation*}
  \sup_{\rmg\in U,\, x\in[\G]}\lvert f(x\rmg)-f(x)\rvert<\varepsilon.
\end{equation*}
We note that the bounded uniformly continuous functions form a subspace $\mcC_{u}([\G])$ of the continuous functions that contains the space $\mcC_{c}([\G])$ of compactly supported continuous functions on $[\G]$.
Equipped with the uniform norm $\mcC_{u}([\G])$ is a Banach algebra.

The natural action of $\G(\adele)$ on the space of bounded measurable functions given by $g\cdot f(x)=f(xg)$ preserves $\mcC_{u}([\G])$ and the resulting representation is strongly continuous. 
In fact, the space $\mcC_{u}([\G])$ can be identified with the space of continuous vectors for the representation of $\G(\adele)$ on $L_m^\infty([\G])$.

\subsubsection{Decomposing $\mcC_{u}([\G])$}

We use the conditional measures for $\mathcal{A}$ to decompose functions on $[\G]$.
For $f$ a bounded measurable function on $[\G]$ define
\begin{align*}
\projchar{f}: x \mapsto \int_{[\G]} f \,\mathrm{d} m_x^\mathcal{A}.
\end{align*}
Note that $\projchar{f}$ is $\G(\Aa)^+$-invariant and $\cA$-measurable.
Moreover, $\projchar{f}$ and $f-\projchar{f}$ are bounded.
The function $f-\projchar{f}$ satisfies that
\begin{align*}
\int_{[\G]} (f-\projchar{f}) \,\mathrm{d} m_x^\mathcal{A} = 0
\end{align*}
since $\projchar{f}$ is constant on $\G(\adele)^+$-orbits.

\begin{remark}\label{rem:relationtoilya}
  The map $f\mapsto \projchar{f}$ is introduced in \cite{IlyaAnnals} as a projection operator from $L_{m}^{2}([\G])$ onto the \emph{character spectrum} of $L^2_m([\G])$, i.e.~the subspace of $\G(\adele)^+$-invariant functions. Note that for functions in the orthogonal complement of the character spectrum Theorem \ref{thm:main-dynamic} follows directly from Theorem \ref{thm:invariance in product} without assuming $\compactopenf$-invariance of the function -- see \cite[Cor.~3.3]{IlyaAnnals} as well as Lemma~\ref{lem:Thm for orth complement} below.
\end{remark}

\begin{lemma}\label{lem:continuityofprojection}
For any $f\in \mcC_{u}([\G])$ we have $\projchar{f}\in \mcC_{u}([\G])$.
Furthermore, if $f$ is invariant under a subgroup $H<\G(\adele)$ then so is $\projchar{f}$.
\end{lemma}

We have thus obtained for any function $f\in \mcC_{u}([\G])$ the decomposition $f= \projchar{f} + (f-\projchar{f})$ into a $\G(\adele)^+$-invariant function and a function with zero integral over all $\G(\adele)^+$-orbits (which respects additional invariance of $f$).

\begin{proof}[Proof of Lemma~\ref{lem:continuityofprojection}]
We first prove that $\projchar{f}$ is uniformly continuous. 
Let $\varepsilon>0$ be given and choose a symmetric neighborhood $U$ of the identity in $\G(\adele)$ for $f$ and $\varepsilon$ as in the definition of uniform continuity. 
Let $x\in[\G]$ and $\rmg\in U$ be arbitrary.
Invariance of $\mathcal{A}$ under the action of $\G(\adele)$, uniform continuity of $f$, and the triangle inquality imply that $\projchar{f}$ is uniformly continuous.
If $f$ is $H$-invariant we see that
\begin{align*}
\projchar{f}(xh) 
&=\int_{[\G]} f \,\mathrm{d}m_{xh}^{\cA} 
=\int_{[\G]}h.f \,\mathrm{d} m_x^{\cA}
=\int_{[\G]}f \,\mathrm{d} m_x^{\cA}
=\projchar{f}(x)
\end{align*}
as claimed.
\end{proof}

\subsection{Concluding the proof of Theorem \ref{thm:main-dynamic}}
We recall that our final goal is to prove that any weak${}^\ast$-limit of the Haar measures on the toral packets in Theorem \ref{thm:main-dynamic} is equal to $m_{[\G]_{\compactopenf}}$.
This amounts to showing that $\int_{[\G]} f \,\mathrm{d} \mu = \int_{[\G]} f \,\mathrm{d} m$ for a large class of $\compactopenf$-invariant functions~$f$ and any such weak${}^\ast$-limit $\mu$.

\begin{lemma}\label{lem:Thm for orth complement}
Let $\mu$ be any measure on $[\G]$ invariant under $\G(\adele)^+$.
Then for any $f \in \mcC_{u}([\G])$ with $\projchar{f}=0$ we have
\begin{align*}
\int_{[\G]} f \,\mathrm{d} \mu = \int_{[\G]} f \,\mathrm{d} m=0.
\end{align*}
\end{lemma}

In the context of Theorem \ref{thm:main-dynamic}, the measure $\mu$ is $\G(\adele)^+$-invariant by Theorem~\ref{thm:invariance in product}.

\begin{proof}
  Let $\{\mu_{x}^{\mathcal{A}}:x\in[\G]\}$ be a family of conditional measures for $\mu$ with respect to $\mathcal{A}$.
  As $\mu$ is $\G(\adele)^{+}$-invariant, the measure $\mu_{x}^{\mathcal{A}}$, which is supported on the orbit $x\G(\adele)^{+}$, is $\G(\adele)^{+}$-invariant for $\mu$-almost all $x\in[\G]$. 
By uniqueness of the invariant measure, it follows that $\mu_{x}^{\mathcal{A}}=m_{x}^{\mathcal{A}}$ for $\mu$-almost all $x\in[\G]$. 
Thus we have shown that
  \begin{equation*}
    \int_{[\G]}f\,\mathrm{d}\mu_{x}^{\mathcal{A}}
    =\int_{[\G]}f\,\mathrm{d}m_{x}^{\mathcal{A}}
	= \projchar{f}(x)    
    =0
  \end{equation*}
  for $\mu$-almost every $x\in[\G]$. In particular, we obtain the statement of the lemma by the characterizing properties of conditional measures.
\end{proof}

%

\begin{proof}[Proof of Theorem \ref{thm:main-dynamic}]
Recall that we have $\compactopenf = \prod_i \compactopen_{\mathrm{f},i}$ where  for any $i$ the compact open subgroup $\compactopen_{\mathrm{f},i}<\G_i(\adelef)$ satisfies \eqref{eq:assumption compactopen}, i.e.~that 
\begin{align*}
\Norm_{\quat_i}: \compactopen_{\mathrm{f},i} \rightarrow \rquot{\widehat{\Z}^{\times}}{(\widehat{\Z}^{\times})^2}
\end{align*}
is surjective.

Let $\mu$ be a weak${}^\ast$-limit of the Haar measures in Theorem~\ref{thm:main-dynamic}.
We need to show that for any $\compactopenf$-invariant compactly supported $f$ on $[\G]$ we have
 \begin{equation}\label{eq:KEquidis}
 \int_{[\G]}f\,\mathrm{d}\mu=\int_{[\G]}f\,\mathrm{d}m.
 \end{equation}
By Theorem \ref{thm:invariance in product} the measure $\mu$ is $\G(\adele)^+$-invariant and hence Lemma \ref{lem:Thm for orth complement} applies to show \eqref{eq:KEquidis} for $f-\projchar{f}$ as $\projchar{(f-\projchar{f})}= 0$.
It thus remains to check \eqref{eq:KEquidis} for $\projchar{f}$ which is $\compactopenf\,\G(\adele)^+$-invariant by Lemma \ref{lem:continuityofprojection}.
Note that $\compactopenf\, \G(\adele)^{+}$ is a subgroup of $\G(\Aa)$ as $\G(\Aa)^+$ is normal.

We  claim that such a function $\projchar{f}$ is constant in which case \eqref{eq:KEquidis} is obvious. This follows from the fact that the double quotient
$$
\lrquot{\G(\bQ)}{\G(\adele)}{\compactopenf\, \G(\adele)^{+}}
$$
is a singleton under our assumptions on $\compactopenf$.
To this end, it suffices to prove that for any $i \in \{0,\ldots,s\}$ the double quotient
\begin{equation}\label{eq:trivialquotient}
\lrquot{\G_i(\bQ)}{\G_i(\adele)}{\compactopen_{\mathrm{f},i}\,\G_i(\adele)^{+}}	
\end{equation}
is a singleton. For this we consider the group homomorphism induced by the reduced norm  $\Nr_i=\Nr_{\quat_i}$
$$
\Nr_i:\G_i(\bQ)\bash\G_i(\adele)/{\compactopen_{\mathrm{f},i}\,\G_i(\adele)^{+}}
 \ra \Qt\bash\Aat/{\Aat}^2\Nr_i(\compactopen_{\mathrm{f},i}).
$$
We note first that it is injective hence an isomorphism onto its image. 

Let us compute the image: for any prime $q$ the norm $\Nr_i:\quatt_i(\Q_q) \to \Q_q^\times$ is surjective since any non-degenerate quadratic form in five variables over $\Q_q$ is isotropic \cite[Ch.~4]{cassels}. Over the real numbers, the norm $\Nr_{i}:\quat_i(\R)^\times \to \R^\times$ is surjective if $\quat_i$ is indefinite and otherwise it has image $\R_{>0}$. 
It follows that
$$\Nr_i(\quatt_i(\Aa))=\begin{cases}\Aat &\hbox{ if $\quat_i$ is indefinite},\\ \Rr_{>0}\Aa_{\mathrm{f}}^\times&\hbox{ if $\quat_i$ is definite}\end{cases}.$$
In addition, by the Hasse-Minkowski Theorem \cite[Ch.~6]{cassels} this also implies that
$$\Nr_i(\quatt_i(\bQ))=\begin{cases}\Qt &\hbox{ if $\quat_i$ is indefinite},\\ \Qq_{>0}&\hbox{ if $\quat_i$ is definite}\end{cases}.$$
From this we conclude that \eqref{eq:trivialquotient} is isomorphic to
$$
\begin{cases}\Qt\bash\Aat/{\adele^\times}^2\Nr_i(\compactopen_{\mathrm{f},i}) &\hbox{ if $\quat_i$ is indefinite},\\ 
\Q_{>0}\bash\R_{>0}\times\Aft/{\adele^\times}^2\Nr_i(\compactopen_{\mathrm{f},i})&\hbox{ if $\quat_i$ is definite.}\end{cases}$$
Under the assumption 
\eqref{eq:assumption compactopen} (i.e.~$\Nr_i(K_{\rmf,i})=\whZt/(\whZt)^2$) this is either
$$\Aat/\Qt\whZt{\adele^\times}^2\hbox{ or }\R_{>0}\Aft/\Q_{>0}\whZt {\adele^\times}^2$$ and 
both are trivial since
$\Aft= \Q_{>0} \whZt$.  This concludes the proof of the above claim and thus also of the theorem.
\end{proof}

\subsection{The case of Eichler orders}\label{sec:Eichlerorder}
In this section we remark that assumption \eqref{eq:assumption compactopen} holds more generally for open-compact subgroups associated to Eichler orders (such as maximal orders). 
This extension is important as it allows to prove joint equidistributions results similar to Theorem \ref{thm:main-ellipticorder} for {\em Heegner points}.

\begin{definition}
	An {Eichler order} in a quaternion algebra $\quat$ is an intersection of two maximal orders.
\end{definition}

Being an Eichler order is a local property. 
Indeed, an order ~$\mathcal{O} \subset \quat(\Q)$ is Eichler if and only if it is everywhere locally Eichler: for every prime $q$,
$\mathcal{O}_q= \mathcal{O} \otimes \Z_q$ is the intersection of two maximal orders in $\quat(\Q_q)$.

Moreover one has the following classification of local Eichler orders in the quaternion algebra $\quat(\Qq_q)$ (see \cite[Prop.~13.3.4,\ Prop.~23.4.3]{voight}):
\begin{itemize}
\item[--] If $\quat$ is ramified at $q$, $\mcO_q	$ is the unique maximal order of $\quat(\Qq_q)$ given by
\begin{equation*}
\mathcal{O}_q = \{x \in \quat(\Q_q): \Norm(x) \in \Z_q\}.
\end{equation*}
\item[--] If $\quat$ is non-ramified at $q$, one has $\quat(\Qq_q)\simeq \Mat_2(\Qq_q)$ and under this identification, there is some $e \geq 0$ such that $\mcO_q$ is conjugate to the order
$$
\left\{ \begin{pmatrix}
a & b \\ q^e c & d
\end{pmatrix}: a,b,c,d \in \Z_q \right\}.
$$
\end{itemize}
The classification implies the following lemma verifying \eqref{eq:assumption compactopen} for the open-compact subgroup $\unitsadelef{\mathcal{O}}$ when $\mathcal{O}$ is an Eichler order.

\begin{lemma}\label{normsurjeichler}
For $\mathcal{O} \subset \quat(\Q)$  an Eichler order we have
\begin{align}\label{eq:norm on Eichler}
\Norm\big(\unitsadelef{\mathcal{O}} \big) = \widehat{\Z}^\times.
\end{align}
\end{lemma}


\section{Further refinements}\label{sec:refinements}\label{relaxnormsurj}
In this section we analyse under which assumptions on sequences of toral packets $[\torus_n \rmg_n]_{\compactopenf}$ Theorem~\ref{thm:main-dynamic} holds for a general compact open subgroup $\Kf=\prod_{i}K_{\rmf,i}\subset\GAf$.

By Lemma~\ref{lem:Thm for orth complement} it is sufficient to check whether \eqref{eq:KEquidis} holds for functions invariant under $K_{\rmf}\G(\adele)^+$, i.e.~functions on the quotient $\Gres/K_\rmf$ where $\Gres$ denotes the abelian group
\begin{align*}
\Gres=\lrquot{\G(\Q)}{\G(\adele)}{\G(\adele)^+}.
\end{align*}
Furthermore, we only have to check \eqref{eq:KEquidis} for $\compactopenf$-invariant characters on $\Gres$.
We start by spelling out what they are.

\subsection{Characters on $\Gres$} 
We have $\Gres = \prod_i \Grescomp{i}$ where $$\Grescomp{i}=\lrquot{\G_i(\Q)}{\G_i(\adele)}{\G_i(\adele)^+}.$$
Recall that in the course of proving Theorem \ref{thm:main-dynamic} we have also established the following lemma.

\begin{lemma}\label{lem:normimage}
For any $i\in \{0,\ldots,s\}$ the reduced norm induces isomorphisms
\begin{equation}\label{eq:quotdefinite}
\Nr_i:\Grescomp{i} \simeq \lrquot{\Q_{>0}}{\R_{>0}\times \adelef^\times}{{\adele^\times}^2 }
\end{equation}
if $\quat_i$ is definite and 
\begin{equation}\label{eq:quotindefinite}
\Nr_i:\Grescomp{i} \simeq 
\Q^\times\bash\adele^\times/{{\adele^\times}^2}.
\end{equation}
if $\quat_i$ is indefinite.
\end{lemma}
Note that the natural homomorphism
\begin{align*}
\lrquot{\Q_{>0}}{\R_{>0}\times \adelef^\times}{{\adele^\times}^2 } \to 
\Q^\times\bash\adele^\times/{{\adele^\times}^2}
\end{align*}
is an isomorphism as any element of $\adele^\times$ can be multiplied by $-1 \in \Q^\times$ if necessary to have a positive real component.
Lemma~\ref{lem:normimage} thus gives an isomorphism (independent of the ramification at the archimedean place)
\begin{equation}\label{eq:eliminating pm1 issue}
\Nr_i:\Grescomp{i} \simeq 
\Q^\times\bash\adele^\times/{{\adele^\times}^2}.
\end{equation}
In particular, $\Gres$ is a compact abelian group, every continuous function on $\Gres$ is uniformly continuous,
and $\mcC(\Gres)$ is densely generated by its group of characters $\what{\Gres}$.

By the previous lemma, any character of $\Grescomp{i}$ is of the shape $\chi_i\circ\Nr_{i}$
for a quadratic Hecke character $\chi_i$ of $\Qq$, i.e.~a character 
$$\chi_i:\Qt\bash\Aat/{\Aat}^2\to \{\pm 1\}.$$ 
Thus, any character $\chi\in\what{\Gres}$ is of the shape
\begin{align}\label{eq:tuplechar}
\rmg=(g_i)_i\in\Gres
\mapsto \chi(\rmg)=\prod_{i=0}^s\chi_i(\Nr_{i}(g_i))
\end{align}
for a uniquely defined tuple $(\chi_i)_i$ of quadratic Hecke characters as above. 
The $\Kf$-invariant characters correspond to the tuples of characters $(\chi_i)_i$ where $\chi_i$ is $\Nr_i(K_{\rmf,i})$-invariant for all $i$.

\subsection{The torus integral for a character}
Let $\chi$ be a character of $\Gres$ given by a tuple $(\chi_i)_i$ of quadratic Hecke characters as in \eqref{eq:tuplechar}.
We view $\chi$ as a $\G(\Aa)^+$-invariant function in $\mcC_u([\G])$.
Let $[\torus_\iota \rmg]\subset[\G]$ be a toral packet with associated quadratic field $\field$. As of Weyl's equidistribution theorem, we consider the torus integral
$$\int_{[\torus_\iota \rmg]}\chi = \int_{[\torus_\iota]}\chi(t \rmg)dt.$$

Since $\chi$ is a character we have
$$\int_{[\torus_\iota]}\chi(t\rmg)dt=\chi(\rmg)\int_{[\torus_\iota]}\chi(t)dt.$$

Let $\Pi$ be the product map
$$\Pi:(\chi_i)_{i}\mapsto \prod_i\chi_i$$
on quadratic Hecke characters.

\begin{proposition}\label{proptorusintegralchar} Let $(\chi_i)_i$ be a tuple of quadratic Hecke characters and let $\chi$ be the associated character on $\Gres$ as in \eqref{eq:tuplechar}.

\begin{itemize}
	\item[--] If $\Pi(\chi_i)_i\equiv 1$, then for any torus orbit $\torus_\iota$ we have
$$\int_{[\torus_\iota]}\chi(t)dt=1.$$
\item[--] 	If $\Pi(\chi_i)_i\not\equiv 1$, let $\field_{\Pi(\chi_i)_i}$ be the  quadratic field corresponding to the Legendre symbol corresponding to the Hecke character ${\Pi(\chi_i)_i}$.
For any torus $\torus_\iota$ with associated quadratic field $\field$, we have
\begin{align*}
\int_{[\torus_\iota]}\chi(t)dt=
\begin{cases}
0 & \text{if }\field\neq \field_{\Pi(\chi_i)_i}\\
1 & \text{else}
\end{cases}.
\end{align*}
\end{itemize}
\end{proposition}

\begin{proof}
Note that for any $x \in \field$ and any $i$ we have $\Nr_i(\iota_i(x)) = \Nr_{\field}(x)$ where $\Nr_\field$ denotes the norm form of $\field/\Q$.
Hence, we have
\begin{align*}
\int_{[\torus_\iota]}\chi(t)dt
=
\int_{[\torus_{\field}]}\prod_{i}\chi_i(\Nr_{i}(\iota_i(t)))dt
=\int_{[\torus_{\field}]} \Pi(\chi_i)_i\big(\Nr_{\field}(t)\big)dt.
\end{align*}
This shows the first part, so we assume that $\Pi(\chi_i)_i$ is non-trivial.

The character $\Pi(\chi_i)_i$ corresponds to a classical Legendre symbol character and is associated with some uniquely defined quadratic field $\field_{\Pi(\chi_i)_i}$. 
If $\field=\field_{\Pi(\chi_i)_i}$, then we have by definition of the Legendre symbol
\begin{align*}
\Pi(\chi_i)_i \circ \Nr_{\field} =1 \implies \int_{[\torus_{\field}]}\Pi(\chi_i)_i(\Nr_{\field}(t))dt = 1.
\end{align*}
If $\field\not=\field_{\Pi(\chi_i)_i}$ there exists a prime $p$ split in $\field$ and inert in $\field_{\Pi(\chi_i)_i}$ so that for $t_p$ in the corresponding idele class one has
\begin{align*}
\Pi(\chi_i)_i\big(\Nr_{\field}(t_p)\big)=-1 \implies \int_{[\torus_{\field}]}\Pi(\chi_i)_i(\Nr_{\field}(t))dt = 0
\end{align*}
as follows from the substitution $t \mapsto t\, t_p$ where $t_p$ is as above.
\end{proof}

\begin{definition} 
Let $\Kf=\prod_{i}K_{\rmf,i}\subset\GAf$ be a compact open subgroup. The set of {\em exceptional fields} attached to $\Kf$ is defined as the set of quadratic fields attached to the non-trivial products $\Pi(\chi_i)_i$ for which the corresponding character $\chi$ of $\Gres$ is $\Kf$ invariant:
$$\mscK(\Kf)
=\{\field_{\Pi(\chi_i)_i}: \Pi(\chi_i)_i\not=1,\ \chi=\prod_i\chi_i\circ\Nr_{\quat_i}\ K_{\rmf}\hbox{-invariant}\ \}.$$
\end{definition}

In particular, given a limit measure $\mu$ of a sequence of toral packets $[\torus_{\iota_n}\rmg_n]$ for which the underlying quadratic field is constant equal to
$\field_{\Pi(\chi_i)_i}$, the measure $\mu$ cannot satisfy \eqref{eq:KEquidis} for the test function $\chi=\prod_i\chi_i\circ\Nr_{\quat_i}$ if the latter is $K_{\rmf}$-invariant.
Indeed, all torus integrals in that case have modulus $1$ as
\begin{align*}
\bigg| \int_{[\torus_{\iota_n}\rmg_n]}\chi \bigg| = |\chi(\rmg_n)| =1
\end{align*}
by Proposition~\ref{proptorusintegralchar}.
On the other hand, since $\chi$ is non-trivial 
$$\int_{[\G]}\chi = \prod_i\int_{[\G_i]}\chi_i(\Nr_{\quat_i}(g_i))dg_i=0.$$
 
Thus, these finitely many exceptional fields are natural obstructions to equidistribution on $[\G]_{\Kf}$. Therefore, one only has to  consider sequences of toral packets $[\torus_{\iota_n}\rmg_n]$ whose associated quadratic fields $\field_n$ are not exceptional.

We remark that $\mscK(\Kf)$ is a finite set as $\Gres/\Kf$ is finite.

\subsection{A refined equidistribution criterion}

We denote by $$\pi_i:\G_i(\Aa)\to \Grescomp{i},\ \pi: \G(\Aa) \to \Gres$$ the natural projections. We also define 
the closed diagonal subgroup
\begin{align*}
\Delta \Gres = \{h \in \Gres: \Nr_i(h_i) = \Nr_j(h_j)\text{ for all }i,j\}.
\end{align*}
Here, the condition $\Nr_i(h_i) = \Nr_j(h_j)$ is understood in the sense of \eqref{eq:eliminating pm1 issue}.

\begin{cor}\label{cor:main theorem for general cptopen}
  Let $\compactopenf = \prod_i\compactopen_{\mathrm{f},i} < \G(\adelef)$ be a compact open subgroup and let $\mscK(K_\rmf)$ be the set of exceptional quadratic fields attached to it.
  Let $q_1,q_2$ be distinct odd primes.
  
  For any sequence $\rmg_n\in\GA$ and any sequence of diagonally embedded tori $\torus_n\subset\G$ attached to quadratic fields $\field_n$ such that
  \begin{itemize}
  \item[--] for any $n$, $\field_n\not\in \mscK(K_\rmf)$,
  \item[--] $\disc([\torus_n\rmg_n])\ra\infty$ for $n\ra\infty$, and
  \item[--] $q_1$ and $q_2$ are split in $\field_n$ for every $n$,
  \end{itemize}
the following are equivalent.
  \begin{enumerate}
    \item The packets $[\torus_{n}\rmg_{n}]_{\compactopenf}$ equidistribute in $[\G]_{\compactopenf}$.
    \item $\Delta\Gres\,\pi(\Kf)= {\Gres}$.
  \end{enumerate}
\end{cor}

\proof By Theorem \ref{thm:invariance in product}, Lemma~\ref{lem:Thm for orth complement},
Proposition \ref{proptorusintegralchar} and our assumption that the fields $\field_n$ are never exceptional, equidistribution  in $[\G]_{\compactopenf}$ is equivalent to the non-existence of a character $\chi$ of $\Gres$ which is non-trivial, $\Kf$-invariant and such that $\Pi(\chi_i)_i$ is trivial. But $\Pi(\chi_i)_i$ being trivial is equivalent to the tuple $(\chi_i\circ \Nr_i)_i$ being constant equal to $1$ along the diagonal subgroup $\Delta\Gres$ and $\chi$ being $\Kf$-invariant is equivalent to $(\chi_i)_i$ being $\pi(\Kf)$-invariant (hence constant equal to $1$ along $\pi(\Kf)$).
\qed




\bibliographystyle{amsalpha}

\begin{thebibliography}{9}



\bibitem{AEShigherdim}
\textit{M.~Aka}, \textit{M.~Einsiedler} and \textit{U.~Shapira},
Integer points on spheres and their orthogonal grids, J. London Math. Soc. \textbf{93} (2016), no.~1, 143--158.

\bibitem{AES3D}
\textit{M.~Aka}, \textit{M.~Einsiedler} and \textit{U.~Shapira},
Integer points on spheres and their orthogonal lattices, Invent. Math. \textbf{206}, (2016), no.~2, 379--396.

\bibitem{AEW-2in4}
\textit{M.~Aka}, \textit{M.~Einsiedler}, and \textit{A.~Wieser},
Planes in four space and four associated {CM} points, preprint, arXiv:1901.05833 (2019).

\bibitem{borelharishchandra}
\textit{A.~Borel} and \textit{Harish-Chandra},
Arithmetic Subgroups of Algebraic Groups, Ann. of Math. \textbf{75} (1962), no.~3, 485--535.

\bibitem{cassels}
\textit{J.W.S.~Cassels},
Rational quadratic forms, London Mathematical Society Monographs vol.~13, Academic Press Inc. (1978).

\bibitem{CU}
\textit{L.~Clozel} and \textit{E.~Ullmo},
\'{E}quidistribution de mesures alg\'{e}briques, Compos. Math. \textbf{141} (2005), no.~5, 1255--1309.

\bibitem{conrad}
\textit{B.~Conrad},
Gross-Zagier revisited, with an appendix by \textit{W.R.~Mann}, in: Heegner points and Rankin $L$-series, Math. Sci. Res. Inst. Publ. vol.~49, Cambridge Univ. Press, Cambridge (2004), 67--163.

\bibitem{Cornut-Invent}
\textit{C.~Cornut},
Mazur's conjucture on higher Heegner points, Invent. Math. \textbf{148} (2002), no.~3, 495--523.

\bibitem{CornutJetchev}
\textit{C.~Cornut} and \textit{D.~Jetchev},
Liftings of reduction maps for quaternion algebras, Bull. Lond. Math. Soc. \textit{45}, (2013), no.~2, 370--386.

\bibitem{CVDoc}
\textit{C.~Cornut} and \textit{V.~Vatsal},
CM points and quaternion algebras, Doc. Math. \textbf{10} (2005), 263--309.

\bibitem{CVLMS}
\textit{C.~Cornut} and \textit{V.~Vatsal},
Nontriviality of Rankin-Selberg $L$-functions and CM points, in: $L$-functions and Galois representations, London Math. Soc. Lecture Note Ser. vol.~320, Cambridge Univ. Press, Cambridge (2007), 121--186.

\bibitem{cox}
\textit{D.A.~Cox},
Primes of the form $x^{2}+ny^{2}$. Fermat, class field theory and complex multiplication, A Wiley-Interscience Publication, John Wiley \& Sons, Inc., New York (1989).

\bibitem{deuring41}
\textit{M.~Deuring},
Die Typen der Multiplikatorenringe elliptischer Funktionenk\"orper, Abh. Math. Sem. Hansischen Univ. \textbf{14} (1941), 197--272.

\bibitem{duke88}
\textit{W.~Duke},
Hyperbolic distribution problems and half-integral weight Maass forms, Invent. Math. \textbf{92} (1988), no.~1, 73--90.

\bibitem{DFI}
\textit{W.~Duke}, \textit{J.~Friedlander} and \textit{H.~Iwaniec},
Bounds for automorphic $L$-functions, Invent. Math. \textbf{112} (1993), 1--8.

\bibitem{EL-joining2}
\textit{M.~Einsiedler} and \textit{E.~Lindenstrauss},
Joinings of higher rank torus actions on homogeneous spaces, Publ. Math. Inst. Hautes \'Etudes Sci. \textbf{129} (2019), 83--127.

\bibitem{ELMV-DukeJ}
\textit{M.~Einsiedler}, \textit{E.~Lindenstrauss}, \textit{Ph.~Michel} and \textit{A.~Venkatesh},
Distribution of periodic torus orbits on homogeneous spaces, Duke Math. J. \textbf{148} (2009), no.~1, 119--174.

\bibitem{dukeforcubic}
 \textit{M.~Einsiedler}, \textit{E.~Lindenstrauss}, \textit{Ph.~Michel} and \textit{A.~Venkatesh}, 
 Distribution of periodic torus orbits and \uppercase{D}uke's theorem for cubic fields, Ann. of Math. (2) \textbf{173} (2011), 815--885.

\bibitem{ELMV-Ens}
\textit{M.~Einsiedler}, \textit{E.~Lindenstrauss}, \textit{Ph.~Michel} and \textit{A.~Venkatesh},
The distribution of closed geodesics on the modular surface and Duke's theorem, Enseign. Math. (2) \textbf{58} (2011), 249--313.

\bibitem{vol1}
\textit{M.~Einsiedler} and \textit{T.~Ward},
Ergodic theory with a view towards number theory, Graduate Texts in Mathematics vol.~259, Springer (2011).

\bibitem{elkiesonoyang}
\textit{N.~Elkies}, \textit{K.~Ono} and \textit{T.~Yang},
Reduction of CM elliptic curves and modular function congruences, Int. Math. Res. Not. IMRN \textbf{44} (2005), 2695--2707.

\bibitem{Linnikthmexpander}
\textit{J.S.~Ellenberg}, \textit{Ph.~Michel} and \textit{A.~Venkatesh},
Linnik's ergodic method and the distribution of integer points on spheres, in: Automorphic representations and $L$-functions, Tata Inst. Fundam. Res. Stud. Math. vol.~22, Tata Inst. Fund. Res., Mumbai (2013), 119--185.

\bibitem{giraud}
\textit{J.~Giraud},
Remarque sur une formule de Shimura-Taniyama, Invent. Math. \textbf{5}, (1968), 231--236.

\bibitem{Gross2}\textit{B.H.~Gross},
Heegner points on {$X_0(N)$}, Modular forms ({D}urham, 1983), 87--105.

\bibitem{Gross}
\textit{B.H.~Gross},
Heights and the special values of $L$-series, in: Number theory (Montreal, Que., 1985), CMS Conf. Proc. vol.~7, Amer. Math. Soc., Providence, RI (1987), 115--187.



\bibitem{HMRL}
\textit{S.~Herrero}, \textit{R.~Menares}, and \textit{J.~Rivera-Letelier},
$p$-Adic Distribution Of CM Points And Hecke Orbits II: Linnik Equidistribution on the Supersingular Locus, preprint (2021).

\bibitem{Iw}
\textit{H.~Iwaniec},
Fourier coefficients of modular forms of half-integral weight, Invent. Math. \textbf{87} (1987), no.~2, 385--401.

\bibitem{Kane09}
\textit{B.~Kane},
CM liftings of supersingular elliptic curves, Jouri. Th\'eor. Nombres Bordeaux \textbf{21} (2009), 635--663.

\bibitem{IlyaAnnals}
\textit{I.~Khayutin},
Joint equidistribution of CM points, Ann. of Math. (2) \textbf{189} (2019), no.~1, 145--276.

\bibitem{IlyaJEMS}
\textit{I.~Khayutin},
Equidistribution on Kuga-Sato varieties of torsion points on CM elliptic curves, J. Eur. Math. Soc. (JEMS), Electronically published on May 25, 2021. doi: 10.4171/JEMS/1067.

\bibitem{lang-elliptic}
\textit{S.~Lang},
Elliptic Functions, 2nd ed., Graduate Texts in Mathematics vol.~112, Springer (1987).

\bibitem{LV}
\textit{K.~Lauter} and \textit{B.~Viray},
On singular moduli for arbitrary discriminants, Int. Math. Res. Not. IMRN, (2015), no.~19, 9206--9250.

\bibitem{LR}
\textit{S.~Lester} and \textit{M.~Radziwi\l\l},
Quantum unique ergodicity for half-integral weight automorphic forms, Duke Math. J. \textbf{169} (2020), no.~2, 279--351.

\bibitem{linnik}
\textit{Yu.V.~Linnik},
Ergodic properties of algebraic fields, Translated from Russian by M.S.~Keane, Ergebnisse der Mathematik und ihrer Grenzgebiete vol.~45, Springer-Verlag, New York (1968).

\bibitem{michel04}
\textit{Ph.~Michel},
The subconvexity problem for Rankin-Selberg $L$-functions and equidistribution of Heegner points, Ann. of Math. (2) \textbf{160} (2004), no.~1, 185--236.

\bibitem{EPSbook}
\textit{Ph.~Michel},
Equidistribution, Periods and Subconvexity (in preparation) (2020).

\bibitem{MVIHES}
\textit{Ph.~Michel} and \textit{A.~Venkatesh},
The subconvexity problem for $\mathrm{GL}_2$, Publ. Math. Inst. Hautes \'{E}tudes Sci. \textbf{111} (2010), 171--271.

\bibitem{milne}
\textit{J.~Milne},
Complex Multiplication, Course Notes, \url{https://www.jmilne.org/math/CourseNotes/cm.html} (2006).

\bibitem{platonov}
\textit{V.~Platonov} and \textit{A.~Rapinchuk},
Algebraic groups and number theory, Translated from the Russian original by R.~Rowen, Pure and Applied Mathematics vol.~139, Academic Press, Inc. (1994).

\bibitem{RS}
\textit{M.~Radziwi\l\l} and \textit{K.~Soundararajan},
Moments and distribution of central {$L$}-values of quadratic twists of elliptic curves, Invent. Math. \textbf{202} (2015), no.~3, 1029--1068.

\bibitem{Ratner-joining}
\textit{M.~Ratner},
Raghunathan's conjectures for Cartesian products of real and $p$-adic Lie groups, Duke Math. J. \textbf{77} (1995), no.~2, 275--382.

\bibitem{SerreCM}
\textit{J.-P.~Serre},
Complex multiplication, in: Algebraic Number Theory (Proc. Instructional Conf., Brighton, 1965), Thompson, Washington, D.C. (1967), 292--296.

\bibitem{serre}
\textit{J.-P.~Serre},
A course in arithmetic, Graduate Texts in Mathematics vol.~7, Springer (1973).

\bibitem{SerreTate_goodreduction}
\textit{J.-P.~Serre} and \textit{J.~Tate}, 
Good reduction of abelian varieties, Ann. of Math. \textbf{88} (1968), no.~3, 492--517.

\bibitem{silverman-advanced}
\textit{J.~Silverman},
Advanced Topics in the Arithmetic of Elliptic Curves, Graduate Texts in Mathematics vol.~151, Springer (1994).

\bibitem{silverman-aec}
\textit{J.~Silverman},
The Arithmetic of Elliptic Curves, 2nd ed., Graduate Texts in Mathematics vol.~106, Springer (2009).

\bibitem{sutherland}
\textit{A.~Sutherland},
Elliptic Curves, MIT Course 18.783, \url{https://math.mit.edu/classes/18.783/2019/lectures.html} (2019).

\bibitem{Vatsal-Invent}
\textit{V.~Vatsal},
Uniform distribution of Heegner points, Invent. Math. \textbf{148} (2002), no.~1, 1--46.

\bibitem{vigneras}
\textit{M.-F.~Vign\'eras},
Arithm\'etique des alg\`ebres de quaternions, Lecture Notes in Mathematics vol.~800, Springer (1980).

\bibitem{voight}
\textit{J.~Voight},
Quaternion algebras, Graduate Texts in Mathematics, Springer (2021).

\bibitem{Waldspurger}
\textit{J.-L.~Waldspurger},
Sur les valeurs de certaines fonctions $L$ automorphes en leur centre de sym\'{e}trie, Compositio Math. \textbf{54} (1985), no.~2, 173--242.

\bibitem{waterhouse}
\textit{W.~Waterhouse},
Abelian varieties over finite fields, Ann. scient. \'Ec. Norm. Sup. 4.Ser \textbf{2} (1969), no.~4, 521--560.

\bibitem{W-Linnik}
\textit{A.~Wieser},
Linnik's problems and maximal entropy methods, Monatsh. Math. \textbf{190} (2019), no.~1, 153--208.

\bibitem{yang08}
\textit{T.~Yang},
Minimal CM liftings of supersingular elliptic curves, Pure Appl. Math. Q. \textbf{4} (2008), no.~4, 1317--1326.

\end{thebibliography}

\end{document}